\documentclass[a4paper]{article}

\usepackage{fullpage}

\usepackage[pdftex,breaklinks,colorlinks,
linkcolor=blue,
citecolor=blue,
urlcolor=blue]{hyperref}

\usepackage{enumerate}

\usepackage{amsmath,amssymb,amsthm}

\newtheorem{thm}{Theorem}[section]
\newtheorem{lem}[thm]{Lemma}
\newtheorem{cor}[thm]{Corollary}
\newtheorem{rmk}[thm]{Remark}

\newcommand{\trans}[1]{\widetilde{#1}}
\newcommand{\bica}{$\beta_1$}
\newcommand{\bicb}{$\beta_2$}

\DeclareMathOperator{\core}{core}
\DeclareMathOperator{\comp}{comp}
\DeclareMathOperator{\bicomp}{bicomp}
\DeclareMathOperator{\ext}{ext}

\title{On the generalized Helly property of hypergraphs,\\ cliques, and bicliques\footnote{Preliminary versions of some of the results in this work were stated (without proofs) in the extended abstract~\cite{DBLP:journals/endm/DouradoGS15}.}}

\author{Mitre C.\ Dourado\footnote{Instituto de Computa\c{c}\~{a}o, Universidade Federal do Rio de Janeiro, Rio de Janeiro, Brazil. E-mail address: \texttt{mitre@ic.ufrj.br}}\qquad Luciano N.\ Grippo\footnote{Instituto de Ciencias, Universidad Nacional de General Sarmiento, Los Polvorines, Buenos Aires, Argentina and Consejo Nacional de Investigaciones Cient\'ificas y T\'ecnicas, Argentina. E-mail address: \texttt{lgrippo@campus.ungs.edu.ar}}\qquad Mart\'\i n D.\ Safe\footnote{Departamento de Matem\'atica, Universidad Nacional del Sur (UNS), Bah\'ia Blanca, Argentina and INMABB, Universidad Nacional del Sur (UNS)-CONICET, Bah\'ia Blanca, Argentina. E-mail address: \texttt{msafe@uns.edu.ar}}}

\date{}

\begin{document}

\maketitle

\begin{abstract}
A family of sets is \emph{$(p,q)$-intersecting} if every nonempty subfamily of $p$ or fewer sets has at least $q$ elements in its total intersection. A family of sets has the \emph{$(p,q)$-Helly property} if every nonempty $(p,q)$-intersecting subfamily has total intersection of cardinality at least $q$. The $(2,1)$-Helly property is the usual \emph{Helly property}. A hypergraph is \emph{$(p,q)$-Helly} if its edge family has the $(p,q)$-Helly property and \emph{hereditary $(p,q)$-Helly} if each of its subhypergraphs has the $(p,q)$-Helly property. A graph is \emph{$(p,q)$-clique-Helly} if the family of its maximal cliques has the $(p,q)$-the Helly property and \emph{hereditary $(p,q)$-clique-Helly} if each of its induced subgraphs is $(p,q)$-clique-Helly. The classes of \emph{$(p,q)$-biclique-Helly} and \emph{hereditary $(p,q)$-biclique-Helly graphs} are defined analogously. In this work, we prove several characterizations of hereditary $(p,q)$-Helly hypergraphs, including one by minimal forbidden partial subhypergraphs. On the algorithmic side, we give an improved time bound for the recognition of $(p,q)$-Helly hypergraphs for each fixed $q$ and show that the recognition of hereditary $(p,q)$-Helly hypergraphs can be solved in polynomial time if $p$ and $q$ are fixed but co-NP-complete if $p$ is part of the input. In addition, we generalize to $(p,q)$-clique-Helly graphs the characterization of $p$-clique-Helly graphs in terms of expansions and give different characterizations of hereditary $(p,q)$-clique-Helly graphs, including one by forbidden induced subgraphs. We give an improvement on the time bound for the recognition of $(p,q)$-clique-Helly graphs and prove that the recognition problem of hereditary $(p,q)$-clique-Helly graphs is polynomial-time solvable for $p$ and $q$ fixed but NP-hard if $p$ or $q$ is part of the input. Finally, we provide different characterizations, give recognition algorithms, and prove hardness results for $(p,q)$-biclique-Helly graphs and hereditary $(p,q)$-biclique-Helly graphs which are analogous to those for $(p,q)$-clique-Helly and hereditary $(p,q)$-clique-Helly graphs.
\end{abstract}

\section{Introduction}

A family $\mathcal F$ of sets is \emph{$p$-wise intersecting} if every nonempty subfamily of $p$ or fewer sets has nonempty total intersection. A family $\mathcal F$ of sets has the \emph{$p$-Helly property} if every nonempty $p$-wise intersecting subfamily of $\mathcal F$ has nonempty total intersection. For example, the celebrated \emph{Helly's theorem}~\cite{Helly} states that any finite family of convex sets in $\mathbb R^{p-1}$ has the $p$-Helly property. The $2$-Helly property is the usual \emph{Helly property}~\cite{MR0357172}.

In this work, we study the more general $(p,q)$-Helly property that originated in the works~\cite{MR1755431} and~\cite{MR1327319}. Let the \emph{core} of a family $\mathcal F$ of sets be the total intersection of the sets of the family. A family of sets is \emph{$p$-wise $q$-intersecting}, or simply \emph{$(p,q)$-intersecting}, if every nonempty subfamily of $\mathcal F$ consisting of $p$ or fewer sets has core of cardinality at least $q$. A family $\mathcal F$ of sets has the \emph{$(p,q)$-Helly property} if every nonempty $(p,q)$-intersecting subfamily of $\mathcal F$ has core of cardinality at least $q$. Clearly, the $(p,1)$-Helly property coincides with the $p$-Helly property.

Let $\mathcal H$ be a hypergraph~\cite{MR0357172}; i.e., $\mathcal H$ is an ordered pair $(X,\mathcal E)$ where $X$ is a finite set and $\mathcal E$ is a finite family of nonempty subsets of $X$ whose union is $X$. The members of $\mathcal E$ are called the \emph{edges} of $\mathcal H$.  Since a hypergraph is uniquely determined by its edge family, we will usually identify a hypergraph with its edge family and apply the concepts defined for families of sets to hypergraphs via this identification. For example, a hypergraph $\mathcal H=(X,\mathcal E)$ has the \emph{$(p,q)$-Helly property} if the family $\mathcal E$ has the $(p,q)$-Helly property. Berge and Duchet~\cite{MR0406801} gave a characterization of $p$-Helly hypergraphs from which a polynomial-time algorithm for the associated recognition problem for each fixed $p$ follows. Their result was extended in~\cite{D-P-S-varsize}, providing a characterization of $(p,q)$-Helly hypergraphs as well as a polynomial-time recognition algorithm for each fixed $p$ and $q$. A faster algorithm for the case where $q=1$ was given in~\cite{MR2457935}. In contrast, the recognition problem of $(p,q)$-Helly hypergraphs is NP-hard if $p$ is part of the input (even if $q$ is fixed)~\cite{MR2223573}, and the problem of determining the computational complexity of recognizing $(2,q)$-Helly hypergraphs when $q$ is part of the input is open~\cite{MR1755431}. For a survey on the computational aspects of the Helly property and its generalizations, the reader is referred to~\cite{D-P-S-ds}.

Let $\mathcal H=(X,\mathcal E)$ be a hypergraph and let $X'\subseteq X$. The \emph{subhypergraph of $\mathcal H$ induced by $X'$} is the hypergraph with vertex set $X'$ and whose edges are those sets $E\cap X'$ that are nonempty as $E$ varies over $\mathcal E$. A hypergraph $\mathcal H$ is \emph{hereditary $(p,q)$-Helly} if each of its subhypergraphs has the $(p,q)$-Helly property. Hereditary $(p,1)$-Helly hypergraphs are simply called \emph{hereditary $p$-Helly}. Characterizations and polynomial-time recognition algorithms for hereditary $p$-Helly hypergraphs for each fixed $p$ were given in~\cite{MR2404220}, where it was also proved that the recognition of hereditary $p$-Helly hypergraphs is co-NP-complete if $p$ is part of the input. An improvement on the time bound for the recognition problem of $p$-Helly graphs for each fixed $p$ was achieved in~\cite{MR2457935}.

A graph is \emph{$(p,q)$-clique-Helly} if the family of its maximal cliques has the $(p,q)$-Helly property. (In this work, the word \emph{maximal} always means inclusion-wise maximal.) The $(p,1)$-clique-Helly graphs are simply called \emph{$p$-clique-Helly}. Dragan~\cite{Dra-thesis} and Szwarcfiter~\cite{MR1447758}, independently, gave a characterization for the class of $2$-clique-Helly graphs (generally known as \emph{clique-Helly graphs}), leading to a polynomial-time recognition algorithm for the class (see also~\cite{MR2321859}). Their characterization was extended in~\cite{MR2365055} to a characterization of $p$-clique-Helly graphs for each $p$ in terms of the so called \emph{expansions}. Moreover, a different characterization for the classes of $(p,q)$-clique-Helly graphs, as well as a polynomial-time recognition algorithm for each fixed $p$ and $q$ were also given in~\cite{MR2365055}. Interestingly, the recognition problem of $(p,q)$-clique-Helly graphs is NP-hard when $p$ or $q$ is part of the input~\cite{MR2365055}.

A graph is \emph{hereditary $(p,q)$-clique-Helly} if each of its induced subgraphs is $(p,q)$-clique-Helly. Hereditary $(p,1)$-clique-Helly graph are called \emph{hereditary $p$-clique-Helly}. Prisner~\cite{MR1238872} gave several characterizations of hereditary $2$-clique-Helly graphs (known as \emph{hereditary clique-Helly graphs}) as well as a polynomial-time recognition algorithm for the class. An alternative recognition algorithm for the same class with better time complexity was given in~\cite{MR2321859}. Characterizations of hereditary $p$-clique-Helly graphs leading to a polynomial-time recognition algorithm for each fixed $p$ were given in~\cite{MR2404220}. In contrast, the recognition of hereditary $p$-clique-Helly graphs is NP-hard if $p$ is part of the input~\cite{MR2404220}.

A \emph{biclique} of a graph is a set of vertices inducing a complete bipartite graph, where we regard edgeless graphs as complete bipartite graphs. A graph is \emph{$(p,q)$-biclique-Helly} if the family of its maximal bicliques has the $(p,q)$-Helly property and \emph{hereditary $(p,q)$-biclique-Helly} if each of its induced subgraphs is $(p,q)$-biclique-Helly. In this work, we show that several results that hold for $(p,q)$-clique-Helly graphs and hereditary $(p,q)$-clique-Helly graphs are mirrored in $(p,q)$-biclique-Helly graphs and hereditary $(p,q)$-biclique-Helly graphs. Two graph classes related to $(2,1)$-biclique-Helly and hereditary $(2,1)$-biclique-Helly graphs were introduced in \cite{MR2365416} and~\cite{MR2383735} under the names of `biclique-Helly' and `hereditary biclique-Helly' graphs, respectively, but their graph classes differ from the ones studied here because in \cite{MR2365416,MR2383735} edgeless graphs are not regarded as complete bipartite graphs.

This work is organized as follows. In the next subsection, we give the basic definitions and notation. In Section~\ref{sec:HpqHelly}, we give improved time bounds (upon those of \cite{D-P-S-varsize}) for the recognition problem of $(p,q)$-Helly hypergraphs for each fixed $q$, prove different characterizations of hereditary $(p,q)$-Helly hypergraphs (including one by minimal forbidden partial subhypergraphs), and also derive a polynomial-time recognition algorithm for hereditary $(p,q)$-Helly hypergraphs for both $p$ and $q$ fixed. In contrast, we prove that the recognition problem of hereditary $(p,q)$-Helly hypergraphs is NP-hard if $p$ is part of the input. In Section~\ref{sec:HpqCHelly}, we give an improvement upon the time bound given in \cite{MR2365055} for the recognition problem of $(p,q)$-clique-Helly graphs, characterize the class of hereditary $(p,q)$-clique-Helly in different ways (including a characterization by forbidden induced subgraphs), and give a polynomial-time recognition for hereditary $(p,q)$-clique-Helly graphs for fixed $p$ and $q$ and prove that the recognition problem of hereditary $(p,q)$-clique-Helly graphs is NP-hard if $p$ or $q$ is part of the input.  In Section~\ref{sec:HpqBH}, we prove analogous characterizations and time bounds for the recognition problem for the classes of $(p,q)$-biclique-Helly graphs and hereditary $(p,q)$-biclique-Helly graphs. Besides, we prove that the recognition problem for both classes is co-NP-complete if $p$ or $q$ is part of the input.

\subsection{Basic definitions and notation}\label{sec:defs}

All graphs in this work are finite, undirected, and without loops or multiple edges. Let $G$ be a graph. We denote by $V(G)$ and $V(E)$ its vertex and edge set, respectively, and by $\overline G$ the complement of $G$. The \emph{neighborhood} of a vertex $v$ of $G$ is denoted by $N_G(v)$ and its \emph{closed neighborhood} $N_G(v)\cup\{v\}$ by $N_G[v]$. A vertex $v$ of $G$ is \emph{universal} if it is adjacent to every other vertex of $G$. Two vertices $u$ and $v$ of $G$ are \emph{false twins} if $N_G(u)=N_G(v)$ (which implies that $u$ and $v$ are nonadjacent). If $W\subseteq V(G)$, we denote by $G[W]$ the subgraph of $G$ induced by $W$ and by $G-W$ the subgraph of $G$ induced by $V(G)-W$. Given a graph $H$, we say $G$ \emph{contains no induced $H$}, or simply \emph{$G$ is $H$-free}, if $G$ has no induced subgraph isomorphic to $H$. A \emph{clique} (resp.\ \emph{stable set}) of a graph is a set of pairwise adjacent (resp.\ nonadjacent) vertices. A \emph{complete graph} is a graph whose vertex set is a clique.  The complete graph on $n$ vertices is denoted $K_n$. A vertex $v$ is \emph{complete} (resp.\ \emph{anticomplete}) to a vertex set $W$ if $v$ is adjacent (resp.\ nonadjacent) to every vertex of $W-\{v\}$. If $W_1$ and $W_2$ are two vertex sets, then $W_1$ is \emph{complete} (resp.\ \emph{anticomplete}) to $W_2$ if each vertex of $W_1$ is complete (resp.\ anticomplete) to $W_2$. A \emph{bipartition} of a graph is a partition of its vertex set into two (possibly empty) stable sets. A graph is \emph{bipartite} if it admits a bipartition and \emph{complete bipartite} if it has a bipartition $\{X,Y\}$ such that $X$ is complete to $Y$. A \emph{component} of a graph is a maximal connected subgraph. We denote by $P_n$ the chordless path on $n$ vertices, respectively. If $t$ is a positive integer, we denote by $tG$ the disjoint union of $t$ graphs each of which isomorphic to $G$.

Let $\mathcal H=(X,\mathcal E)$ be a hypergraph. We denote its \emph{vertex set} $X$ by $V(\mathcal H)$ and its \emph{edge family} $\mathcal E$ by $E(\mathcal H)$. The hypergraph $\mathcal H$ is \emph{empty} when $X$ is empty. A \emph{partial hypergraph} of $\mathcal H$ is any hypergraph $\mathcal H'$ such that $E(\mathcal H')$ is a subfamily of $E(\mathcal H)$. A \emph{partial subhypergraph} of $\mathcal H$ is a partial hypergraph of some subhypergraph of $\mathcal H$ or, equivalently, a subhypergraph of a partial hypergraph of $\mathcal H$.
If $E$ is an edge of $\mathcal H$, we denote by $\mathcal H-E$ the hypergraph whose edge family is $E(\mathcal H)$ with the set $E$ removed once. We identify each hypergraph with its edge family. For instance, the \emph{core} of $\mathcal H$, denoted $\core(\mathcal H)$, is the total intersection of $E(\mathcal H)$, and $\mathcal H$ is \emph{$(p,q)$-intersecting} if $E(\mathcal H)$ is $(p,q)$-intersecting.

Let $S$ be a set. We denote the cardinality of $S$ by $\vert S\vert$. We say that a set $S$ is a \emph{$k$-set} if $\vert S\vert=k$, a \emph{$k^-$-set} if $\vert S\vert\leq k$, and a \emph{$k^+$-set} if $\vert S\vert\geq k$. This notation will be applied to any term standing for a set; e.g., cores, subsets, cliques, bicliques, etc. By a \emph{$k$-hypergraph}, a \emph{$k^-$-hypergraph}, and a \emph{$k^+$-hypergraph}, we mean a hypergraph whose edge family consists of exactly $k$, at most $k$, and at least $k$ edges, respectively.

Along this work, $p$ and $q$ denote positive integers and $s$ a nonnegative integer. Graphs are assumed to be stored as adjacency lists. In time complexity analyses, $n$ and $m$ always refer to the number of vertices and edges of the input graph or hypergraph. By $\omega$ we denote the \emph{clique number} (the maximum cardinality of a clique) and by $\psi$ the \emph{biclique number} (the largest cardinality of a biclique), whereas $r$ denotes the \emph{rank} (the largest cardinality of an edge), $\Delta$ the \emph{maximum degree} (maximum number of edges sharing one vertex in common) of the input graph, and $M$ the \emph{total size} (sum of the cardinalities of all the edges) of the input hypergraph.

More specific definitions will be introduced as needed. For any undefined graph or hypergraph notions, the reader is referred to the books by West~\cite{MR1367739} and Berge~\cite{MR0357172}, respectively.

\section{The $(p,q)$-Helly property of hypergraphs}\label{sec:HpqHelly}

In this section, we study the problems of characterizing and recognizing $(p,q)$-Helly hypergraphs and hereditary $(p,q)$-Helly hypergraphs. In first place, we give an improvement upon the time bound given in \cite{D-P-S-varsize} for the recognition of $(p,q)$-Helly hypergraphs for each fixed $q$ by generalizing the algorithm for recognizing $p$-Helly hypergraphs given in \cite{MR2457935}. Afterwards, we give different characterizations of hereditary $(p,q)$-Helly hypergraphs, including one by minimal forbidden partial subhypergraphs, and derive a polynomial-time algorithm for the associated recognition problem for each fixed $p$ and $q$. In contrast, we show that, even for fixed $q$, the recognition of hereditary $(p,q)$-Helly hypergraphs is co-NP-complete if $p$ is part of the input. These results for hereditary $(p,q)$-Helly hypergraphs generalize analogous results for hereditary $p$-Helly hypergraphs proved in~\cite{MR2457935} and \cite{MR2404220} (see Theorems~\ref{thm:HpHelly} and \ref{thm:HpHelly-rec}).

\subsection{$(p,q)$-Helly hypergraphs}\label{ss:pqHH}

We first recall the characterization of $p$-Helly hypergraphs due to Berge and Duchet~\cite{MR0406801}. If $\mathcal H$ is a hypergraph and $P$ is a subset of $V(\mathcal H)$, we denote by $\mathcal H_P$ the hypergraph formed by the edges of $\mathcal H$ containing $P$ and by $\mathcal H_P^\cup$ the hypergraph formed by the edges of $\mathcal H$ containing all but at most one vertex in $P$. Hence, if $P=\{p_1,\ldots,p_\ell\}$, then $\mathcal H_P^\cup$ is the partial hypergraph of $\mathcal H$ consisting precisely of those edges of $\mathcal H$ that are also edges of least one of $\mathcal H_{P_1},\ldots,\mathcal H_{P_\ell}$, where $P_i=P-\{p_i\}$ for each $i\in\{1,\ldots,\ell\}$.

\begin{thm}[\cite{MR0406801}]\label{thm:pHelly} If $p$ is a positive integer, then a hypergraph $\mathcal H$ is $p$-Helly if and only if, for every $(p+1)$-subset $P$ of $V(\mathcal H)$, either $\mathcal H_P^\cup$ is empty or has nonempty core.\end{thm}

For notational convenience, if $\mathcal H$ is a hypergraph, we call any family $\mathcal S$ of $p+1$ pairwise different $q$-subsets of $V(\mathcal H)$, a \emph{$(p+1,q)$-basis} of $\mathcal H$. Each union of all but exactly one member of $\mathcal S$ will be called a \emph{support set} of $\mathcal S$. We denote by $\mathcal H_{\mathcal S}^\cup$ the partial hypergraph of $\mathcal H$ formed by those edges which contain some support set of $\mathcal S$ each. The above characterization of $p$-Helly hypergraphs was extended to $(p,q)$-Helly hypergraphs in~\cite{D-P-S-varsize} as follows.

\begin{thm}[\cite{D-P-S-varsize}]\label{thm:pqHelly}
If $p$ and $q$ are positive integers, then a hypergraph $\mathcal H$ is $(p,q)$-Helly if and only if, for every $(p+1,q)$-basis $\mathcal S$ of $\mathcal H$, either $\mathcal H_\mathcal S^\cup$ is empty or has a $q^+$-core.\end{thm}

We now derive two variants of the above characterization for future reference. We say that a $(p+1,q)$-basis $\mathcal S$ of a hypergraph $\mathcal H$ is \emph{nontrivial} if each support set of $\mathcal S$ is contained in some edge of $\mathcal H$, and \emph{trivial} otherwise. We observe the following straightforward fact about trivial bases.

\begin{rmk}\label{rmk:nontrivial} If $\mathcal S$ is trivial basis of a hypergraph $\mathcal H$, then $\mathcal H_\mathcal S^\cup$ is empty or its core contains some member of $\mathcal S$.\end{rmk}

The above observation leads to the following variant of Theorem~\ref{thm:pqHelly}.

\begin{cor}\label{cor:pqHelly} If $p$ and $q$ are positive integers, then a hypergraph $\mathcal H$ is $(p,q)$-Helly if and only if, for every nontrivial $(p+1,q)$-basis $\mathcal S$ of $\mathcal H$, $\mathcal H_\mathcal S^\cup$ has a $q^+$-core.\end{cor}
\begin{proof} It follows immediately from Theorem~\ref{thm:pqHelly} because $\mathcal H_\mathcal S^\cup$ is nonempty for any nontrivial $(p+1,q)$-basis $\mathcal S$ of $\mathcal H$ and Remark~\ref{rmk:nontrivial} implies that for each trivial $(p+1,q)$-basis $\mathcal S$ of $\mathcal H$ either $\mathcal H^\cup_\mathcal S$ is empty or has a $q^+$-core.\end{proof}

We will now derive a different variant of Theorem~\ref{thm:pqHelly}. Let $q$ be a positive integer. For any set $S$, we denote by $\varphi_q(S)$ the set of $q$-subsets of $S$. In particular, $\varphi_q(S)\neq\emptyset$ if and only if $S$ has cardinality at least $q$. For every hypergraph $\mathcal H$, we define $\varPhi_q(\mathcal H)$ as the hypergraph whose vertices are the $q$-subsets of $V(\mathcal H)$ that are contained in some edge of $\mathcal H$ and its edge family consists of the sets $\varphi_q(E)$ as $E$ varies over the $q^+$-edges of $\mathcal H$. This operator $\varPhi_q$ mirrors the homonymous operator for graphs defined in \cite{MR2365055} that we will discuss in Section~\ref{sec:HpqCHelly} (see~Theorems~\ref{thm:Phi}, \ref{thm:CPP}, and~\ref{thm:HpqCHelly}). For the time being, we observe the following immediate fact.

\begin{rmk}\label{rmk:phiq} If $\mathcal H$ is a hypergraph and $q$ is a positive integer, then $\varphi_q(\core(\mathcal H))\subseteq\core(\varPhi_q(\mathcal H))$. If, in addition, all the edges of $\mathcal H$ are $q^+$-sets, then $\varphi_q(\core(\mathcal H))=\core(\varPhi_q(\mathcal H))$ and, in particular, $\vert\core(\varPhi_q(\mathcal H))\vert=\binom{\vert\core(\mathcal H)\vert}q$.\end{rmk}

Notice that the set $\varphi_q(\core(\mathcal H))$ can be a proper subset of $\core(\varPhi_q(\mathcal H))$ in case not all edges of $\mathcal H$ are $q^+$-edges. For example, if $E(\mathcal H)=\{\{1\},\{1,2\}\}$ and $q=2$, then $\varphi_q(\core(\mathcal H))=\varphi_2(\{1\})=\emptyset$, whereas $\core(\varPhi_q(\mathcal H))=\{1,2\}$ because $\{1,2\}$ is the only edge of $\varPhi_q(\mathcal H)$.

The second variant of Theorem~\ref{thm:pqHelly} is the following.

\begin{cor}\label{cor:var2} If $p$ and $q$ are positive integers, then a hypergraph $\mathcal H$ is $(p,q)$-Helly if and only if $\varPhi_q(\mathcal H)$ is $p$-Helly.\end{cor}
\begin{proof} Suppose that $\mathcal H$ is $(p,q)$-Helly. Let $\mathcal S$ be any $(p+1)$-subset of $V(\varPhi_q(\mathcal H))$. Thus, $\mathcal S$ is a $(p+1,q)$-basis of $\mathcal H$. Since $\mathcal H$ is $(p,q)$-Helly, Theorem~\ref{thm:pqHelly} implies that $\mathcal H_\mathcal S^\cup$ is empty or has a $q^+$-core. Notice that $\varPhi_q(\mathcal H_\mathcal S^\cup)=(\varPhi_q(\mathcal H))_\mathcal S^\cup$. Thus, if $\mathcal H_\mathcal S^\cup$ is empty, then $(\varPhi_q(\mathcal H))_\mathcal S^\cup$ is empty. Otherwise, $\mathcal H_\mathcal S^\cup$ has a $q^+$-core and thus $(\varPhi_q(\mathcal H))_\mathcal S^\cup$ has nonempty core because $\emptyset\neq\varphi_q(\core(\mathcal H_\mathcal S^\cup))\subseteq\core(\varPhi_q(\mathcal H_\mathcal S^\cup))=\core((\varPhi_q(\mathcal H))_\mathcal S^\cup)$ (by relying on Remark~\ref{rmk:phiq}). Therefore, $\varPhi_q(\mathcal H)$ is $p$-Helly by virtue of Theorem~\ref{thm:pHelly}.

Conversely, suppose that $\varPhi_q(\mathcal H)$ is $p$-Helly. Let $\mathcal S$ be any nontrivial $(p+1,q)$-basis of $\mathcal H$. In particular, each member of $\mathcal S$ is contained in some edge of $\mathcal H$; i.e., $\mathcal S$ is a $(p+1)$-subset of $V(\varPhi_q(\mathcal H))$. Since $\varPhi_q(\mathcal H)$ is $p$-Helly, Theorem~\ref{thm:pHelly} implies that $(\varPhi_q(\mathcal H))_\mathcal S^\cup$ is empty or has nonempty core. Thus, since $\varPhi_q(\mathcal H_\mathcal S^\cup)=(\varPhi_q(\mathcal H))_\mathcal S^\cup$ and $\varPhi_q(\mathcal H_\mathcal S^\cup)$ is nonempty (because $\mathcal S$ is nontrivial), necessarily $\varPhi_q(\mathcal H_\mathcal S^\cup)$ has nonempty core. Hence, as every edge of $\mathcal H_\mathcal S^\cup$ is a $q^+$-set, Remark~\ref{rmk:phiq} implies $\varphi_q(\core(\mathcal H_\mathcal S^\cup))=\core(\varPhi_q(\mathcal H_\mathcal S^\cup))\neq\emptyset$. Thus, $\mathcal H_\mathcal S^\cup$ has a $q^+$-core. Therefore, $\mathcal H$ is $(p,q)$-Helly by virtue of Corollary~\ref{cor:pqHelly}.\end{proof}

As a corollary of Theorem~\ref{thm:pqHelly}, a recognition algorithm for $(p,q)$-Helly hypergraphs that runs in polynomial-time for fixed $p$ and $q$ was given in~\cite{D-P-S-varsize}. A faster algorithm for the more restricted problem of recognizing $p$-Helly hypergraphs was proposed in~\cite{MR2457935}. Recall that the recognition of $(p,q)$-Helly hypergraphs is NP-hard if $p$ is part of the input~\cite{MR2223573}. These results are summarized as follows.

\begin{thm}[\cite{MR2457935,MR2223573,D-P-S-varsize}]\label{thm:pH-rec} The recognition problem for $(p,q)$-Helly hypergraphs:
\begin{enumerate}[(i)]
 \item\label{it:pH-b1} can be solved in $O(m(n+pq)n^{(p+1)q})$ time, where $p$ and $q$ are part of the input;
 \item\label{it:pH-b2} can be solved in $O(prn^{p+1}+Mn^p)$ time if $q=1$.
 \item\label{it:pH-b3} is NP-hard if $p$ is part of the input, even if $q$ is fixed.
\end{enumerate}\end{thm}

Below, we give an improvement upon bound (\ref{it:pH-b1}) of Theorem~\ref{thm:pH-rec} for each fixed $q$ by a generalization of the strategy used in \cite{MR2457935} for obtaining bound (\ref{it:pH-b2}) above. The recognition of $(1,q)$-Helly hypergraphs is based on the observation below, which follows easily from the definition.

\begin{rmk}\label{rmk:1qHelly} If $q$ is a positive integer, then the following statements are equivalent for each hypergraph $\mathcal H$:
\begin{enumerate}[(i)]
 \item\label{it:1qH1} $\mathcal H$ is $(1,q)$-Helly;
 \item\label{it:1qH2} the family of all the $q^+$-edges of $\mathcal H$ is empty or has a $q^+$-core;
 \item\label{it:1qH3} the family $\{\core(\mathcal H_P):\,P\text{ is a $q$-subset of some edge of }\mathcal H\}$ is empty or has $q^+$-core.
\end{enumerate}\end{rmk}

The next lemma will be useful for the recognition problem of the $(p,q)$-Helly property all along this work. Recall that, in time complexity bounds, $r$ stands for the rank of the input hypergraph.

\begin{lem}\label{lem:technical} Let $q$ be any fixed positive integer. Given a positive integer $p$ together with a list of $q$-sets $S_1,\ldots,S_N$ of vertices of some hypergraph $\mathcal H$ such that $S_1,\ldots,S_N$ includes (but is not necessarily limited to) all possible $q$-subsets of the edges of $\mathcal H$, it can be decided whether $\mathcal H$ is $(p,q)$-Helly in $O\bigl((rN+f+n)\binom Np\bigr)$ time assuming that, given any subset $P$ of $V(\mathcal H)$, it can be decided whether $\mathcal H_P$ is empty and, if it is not, also compute $\core(\mathcal H_P)$ in $O(f)$ time (where $f$ stands for some function of $\mathcal H$).\end{lem}
\begin{proof} By virtue of Corollary~\ref{cor:pqHelly}, $\mathcal H$ is $(p,q)$-Helly if and only if $\mathcal H^\cup_\mathcal S$ has a $q^+$-core for each nontrivial $(p+1,q)$-basis $\mathcal S=\{S_{j_1},\ldots,S_{j_{p+1}}\}$ such that $1\leq j_1<\cdots<j_{p+1}\leq N$. In particular, if $N\leq p$, then $\mathcal H$ is $(p,q)$-Helly. Thus, we assume, without loss of generality, that $N>p$. Notice that if $q>r$, then $\mathcal H$ is trivially $(p,q)$-Helly because the only $(p,q)$-intersecting subfamily of $E(\mathcal H)$ is the empty one. Hence, we also assume, without loss of generality, that $q\leq r$. 

We proceed in two stages. In the first stage, we perform a backtracking over all the tuples $(i_1,\ldots,i_k)$ such that $1\leq i_1<\cdots<i_k\leq N-p+k$ and $0\leq k\leq p$. We conduct and also store the result of the whole backtracking over a trie $T_1$, where each tuple $(i_1,\ldots,i_k)$ is represented by a node of $T_1$ corresponding to the ``word'' formed by the sequence of $k$ ``letters'' $i_1,i_2,\ldots,i_k$. Notice that the leaf nodes of $T_1$ represent precisely the tuples for which $k=p$. We identify each tuple $(i_1,\ldots,i_k)$ with the node of $T_1$ representing it. The children of a node $(i_1,\ldots,i_k)$ of $T_1$ (where $k<p$) are the nodes $(i_1,\ldots,i_k,i_{k+1})$ where $i_{k+1}$ ranges from $i_k+1$ to $N-p+k+1$; we assume that these children are stored in a linked list sorted in increasing order of $i_{k+1}$. Backtracking on the set of tuples corresponds to performing a depth-first search on $T_1$. Each time we visit a node $(i_1,\ldots,i_k)$, we update the union $P=S_{i_1}\cup\cdots\cup S_{i_k}$, which can be easily accomplished in $O(q)$ time (by representing $P$ as a vector of length $n$ where we keep track of the number of occurrences of each vertex of $\mathcal H$ among $S_{i_1}\cup\cdots\cup S_{i_k}$). Moreover, whenever we reach a leaf node $(i_1,\ldots,i_p)$, we compute and store at the node the following information: whether $\mathcal H_P$ is empty and, if it is not, also $\core(\mathcal H_P)$. Computing and storing this information takes $O(n+f)$ time per leaf node (taking into account the $O(n)$ time needed to retrieve the elements of $P$ from its vector representation). Since the number of nodes of $T_1$ is $\sum_{k=0}^p\binom{N-p+k}k=\frac{N+1}{N-p+1}\binom Np=O\bigl(N\binom Np\bigr)$ and the number of leaf nodes is $\binom Np$, the total time required for the first stage is $O\bigl((rN+f+n)\binom Np\bigr)$ because $q\leq r$.

For the second stage, we perform a second backtracking over all the tuples $(j_1,\ldots,j_k)$ such that $1\leq j_1<\cdots<j_k\leq N-(p+1)+k$ and $1\leq k\leq p+1$. Similarly to what we did in the first stage, we interpret this backtracking as performing a depth-first search over a second trie $T_2$ whose nodes are identified with this second set of tuples $(j_1,\ldots,j_k)$. However, we can just traverse $T_2$ implicitly (i.e., there is no need to store it explicitly in memory). At all times during the depth-first search over $T_2$, we keep $p+2$ iterators $I_0,I_1,\ldots,I_{p+1}$ pointing to nodes of the trie $T_1$. These iterators are updated each time we visit a node $(j_1,\ldots,j_k)$ of $T_2$ so as to ensure that iterator $I_0$ points to node $(j_1,\ldots,j_{\min\{k,p\}})$ of $T_1$, while iterator $I_\ell$ points to node $(j_1,\ldots,j_{\ell-1},j_{\ell+1},\ldots,j_k)$ of $T_1$ for each $\ell\in\{1,\ldots,k\}$. We now show that updating the iterators can be accomplished in $O(p)$ time per node. We consider three cases, according the each kind of movement during the depth-first search over $T_2$:
\begin{itemize}
 \item \emph{when moving from a node $(j_1,\ldots,j_k)$ to its first children $(j_1,\ldots,j_k,j_k+1)$:} each of $I_1,\ldots,I_{k-1}$ is pointed to the first children of the node it was pointing to, $I_k$ is pointed to the next sibling of the node pointed by $I_0$, $I_{k+1}$ is pointed to the node pointed by $I_0$, and then $I_0$ is either pointed to the first children of the node it was pointing to if $k<p$ or left unchanged if $k=p$; 
 \item \emph{when moving from a node $(j_1,\ldots,j_k)$ to its next sibling $(j_1,\ldots,j_k+1)$:} $I_0$ is either pointed to the next sibling of the node it was pointing to if $k\leq p$ or left unchanged if $k=p+1$, each of $I_1,\ldots,I_{k-1}$ is pointed to the next sibling of the node it was pointing to, and $I_k$ is left unchanged;
 \item \emph{when moving from a node $(j_1,\ldots,j_{k+1})$ having no next sibling to its parent $(j_1,\ldots,j_k)$:} $I_0$ is either pointed to the parent of the node it was pointing to if $k<p$ or left unchanged if $k=p$, and each of $I_1,\ldots,I_k$ is pointed to the parent of the node it was pointing to.
\end{itemize}
Notice that when we visit a leaf node $(j_1,\ldots,j_{p+1})$ of $T_2$, if $\mathcal S=\{S_{j_1},\ldots,S_{j_{p+1}}\}$ and $P_1,\ldots,P_{p+1}$ are its support sets, where $P_\ell=S_{j_1}\cup\cdots S_{j_{\ell-1}}\cup S_{j_{\ell+1}}\cup\cdots\cup S_{j_{p+1}}$ for each $\ell\in\{1,\ldots,p+1\}$, then the iterator $I_\ell$ allows us to know whether $\mathcal H_{P_\ell}$ is empty and, if it is not, also to have access to $\core(\mathcal H_{P_\ell})$ in $O(r)$ time; hence, in $O(pr)$ time, we can decide whether $\mathcal S$ is nontrivial (precisely when none of $\mathcal H_{P_1},\ldots,\mathcal H_{P_{p+1}}$ is empty) and, if so, verify whether $\core(\mathcal H^\cup_\mathcal S)$ has a $q^+$-core (because $\core(\mathcal H^\cup_\mathcal S)$ can be determined as the intersection of the cores of $\mathcal H_{P_1},\ldots,\mathcal H_{P_{p+1}}$). As discussed in the first paragraph of this proof, Corollary~\ref{cor:pqHelly} implies that $\mathcal H$ is $(p,q)$-Helly if and only if all of these verifications succeed. Moreover, since the total number of nodes of $T_2$ is $\sum_{k=0}^{p+1}\binom{N-(p+1)-k}{k}=\frac{N+1}{p+1}\binom Np=O\bigl(\frac Np\binom Np\bigr)$, the whole second stage takes $O\bigl((p+pr)\frac Np\binom Np\bigr)=O\bigl(rN\binom Np\bigr)$ time. Therefore, the whole recognition algorithm takes $O\bigl((rN+f+n)\binom Np\bigr)$ time, as claimed.\end{proof}

The above statements lead to the following time bounds for the recognition of $(p,q)$-Helly hypergraphs.

\begin{thm}\label{thm:rec-pqHelly} If $q$ is any fixed positive integer, then the recognition problem for $(p,q)$-Helly hypergraphs, where $p$ is part of the input, can be solved in:
\begin{enumerate}[(i)]
 \item\label{it:1} $O(M)$ time if $p=1$;
 \item\label{it:2} $O\bigl((rN+M)\binom Np\bigr)$ time if $p\geq 2$, where $N=\binom nq$.
\end{enumerate}
In particular, if $p$ and $q$ are both fixed, then the bound (\ref{it:2}) above becomes $O(rn^{(p+1)q}+Mn^{pq})$.
\end{thm}
\begin{proof} Let $\mathcal H$ be the input hypergraph. As observed in the preceding proof, if $q>r$, then $\mathcal H$ is trivially $(p,q)$-Helly. As whether $q>r$ holds can be decided in $O(M)$ time, we assume, without loss of generality, that $q\leq r$. Bound (\ref{it:1}) is a consequence of the equivalence (\ref{it:1qH1})${}\Leftrightarrow{}$(\ref{it:1qH2}) of Remark~\ref{rmk:1qHelly}, whereas bound (\ref{it:2}) follows directly from Lemma~\ref{lem:technical} because we can enumerate all the $q$-subsets $S_1,\ldots,S_N$ of $V(\mathcal H)$ in $O(qN)$ time, where $N=\binom nq$, and for any given subset $P$ of $V(\mathcal H)$ we can decide whether $\mathcal H_P$ is empty and, if it is not, compute its core in $O(M)$ time. \end{proof}

To see that the above theorem leads to an improvement upon the bound (\ref{it:pH-b1}) of Theorem~\ref{thm:pH-rec} for each fixed $q$, notice that $(rN+M)\binom Np=O(mn^{(p+1)q+1})$ because $r\leq n$, $N\leq n^q$, $M\leq mn$, and $\binom Np=O(n^{pq})$. Since when $q=1$ our algorithm essentially coincides with that of \cite{MR2457935}, the fact that our result also improves the bound (\ref{it:pH-b2}) of Theorem~\ref{thm:pH-rec} is due to the complexity analysis in~\cite{MR2457935} not being as tight as ours. The purpose of making our analysis tighter is that it will then allow us to derive tighter time complexities for other recognition algorithms in what follows.

\subsection{Hereditary $(p,q)$-Helly hypergraphs}

We now turn to the problem of characterizing and recognizing hereditary $(p,q)$-Helly hypergraphs. We begin by introducing some definitions in order to present the results on hereditary $p$-Helly hypergraphs proved in~\cite{MR2457935} and~\cite{MR2404220}. Let $\mathcal H$ be a hypergraph. We say $\mathcal H$ is \emph{strong $p$-Helly}~\cite{MR782622} if, for every nonempty partial hypergraph $\mathcal H'$ of $\mathcal H$, there exist $p$ or fewer edges of $\mathcal H'$ whose core equals that of $\mathcal H'$. An \emph{incidence matrix} $M(\mathcal H)$ of $\mathcal H$ is a $(0,1)$-matrix having one row for each edge and one column for each vertex of $\mathcal H$ and such that there is a $1$ in the intersection of a row and a column if and only if the corresponding edge contains the corresponding vertex. Clearly, the incidence matrix of a hypergraph is unique up to permutation of its rows and/or columns. A \emph{complement of a permutation matrix} is a $(0,1)$-matrix having exactly one $0$ per row and per column. The \emph{complete $r$-uniform hypergraph on $n$ vertices}, denoted $\mathcal K_n^{r}$, is the hypergraph whose edges are all the $r$-subsets of an $n$-set. The following characterizations of hereditary $p$-Helly hypergraphs were proved in~\cite{MR2457935} and~\cite{MR2404220}.

\begin{thm}[\cite{MR2457935,MR2404220}]\label{thm:HpHelly}
If $p$ is a positive integer, then the following statements are equivalent for each hypergraph $\mathcal H$:
\begin{enumerate}[(i)]
\item\label{it:HpH1} $\mathcal H$ is hereditary $p$-Helly;
\item $\mathcal H$ is $(p,q)$-Helly for every $q$;
\item\label{it:HpH3} $\mathcal H$ is strong $p$-Helly;
\item\label{it:HpH4} every partial $(p+1)$-hypergraph of $\mathcal H$ is strong $p$-Helly;
\item\label{it:HpH5} $M(\mathcal H)$ contains no $(p+1)\times(p+1)$ complement of a permutation matrix as a submatrix;
\item no partial subhypergraph of $\mathcal H$ is isomorphic to $\mathcal K_{p+1}^{p}$;
\item\label{it:HpH7} for every $(p+1)$-subset $P$ of $V(\mathcal H)$, either $\mathcal H_P^\cup$ is empty or $\core(\mathcal H_P^\cup)\cap P\neq\emptyset$.
\end{enumerate}
\end{thm}

From the above theorem, polynomial-time recognition algorithms for fixed $p$ follow, whereas the recognition problem was shown to be co-NP-complete if $p$ is part of the input.

\begin{thm}[\cite{MR2457935,MR2404220}]\label{thm:HpHelly-rec} The recognition problem for hereditary $p$-Helly hypergraphs, where $p$ is part of the input:
\begin{enumerate}[(i)]
 \item can be solved in $O(p^2rm^{p+1})$ time;
 \item can be solved in $O(Mn^p+prn^{p+1})$ time;
 \item\label{it:HpH-rec3} is co-NP-complete.
\end{enumerate}\end{thm}
Recognition algorithms for hereditary $2$-Helly hypergraphs having $O(rm^3)$ and $O(r^2\Delta m)$ time complexities were devised in~\cite{MR1082681} and \cite{MR1866814}, respectively.

We will now extend Theorems~\ref{thm:HpHelly} and \ref{thm:HpHelly-rec} to the class of hereditary $(p,q)$-Helly hypergraphs. We introduce the following generalization of the strong $p$-Helly property. We say that a hypergraph $\mathcal H$ is \emph{strong $(p,q)$-Helly} if, for every nonempty $(p,q)$-intersecting partial hypergraph $\mathcal H'$ of $\mathcal H$, there is some nonempty subfamily of $p$ or fewer edges of $\mathcal H'$ whose core equals the core of $\mathcal H'$. Observe that the strong $(p,1)$-Helly property coincides with the strong $p$-Helly property. (In fact, a partial hypergraph $\mathcal H'$ that is not $(p,1)$-intersecting has a nonempty subfamily of $p$ or fewer edges whose core is empty and necessarily $\core(\mathcal H')$ is also empty.)

We also introduce a generalization of $\mathcal K_{p+1}^p$. Let $p$ and $q$ be positive integers and let $s\in\{0,\ldots,q-1\}$. We define $\mathcal J_{p+1,q,s}$ as the unique $(p+1)$-hypergraph $\mathcal H$ (up to isomorphism) having $(p+1)(q-s)+s$ vertices and such that there are $p+1$ pairwise disjoint $(q-s)$-subsets $T_1,\ldots,T_{p+1}$ of $V(\mathcal H)$ such that $E(\mathcal H)=\{V(\mathcal H)-T_i:\,1\leq i\leq p+1\}$. Clearly, $\mathcal J_{p+1,q,s}$ is $(p,q)$-intersecting and has $s$-core. Notice that if $q=1$, then $s=0$ and $\mathcal J_{p+1,1,0}$ coincides with $\mathcal K_{p+1}^p$.

Let $\mathcal S$ be a $(p+1,q)$-basis of a hypergraph $\mathcal H$. We say that $\mathcal S$ is \emph{starlike} if every vertex of $\mathcal H$ which belongs to at least two members of $\mathcal S$ also belongs to $\core(\mathcal S)$. We define the \emph{exterior} of $\mathcal S$, denoted $\ext(\mathcal S)$, as the set of vertices of $\mathcal H$ that belong to some set of $\mathcal S$ but not to $\core(\mathcal S)$. Equivalently, $\mathcal S$ is starlike if and only if there is some $s\in\{0,\ldots,q-1\}$, some $s$-subset $Z$ of $V(\mathcal H)$, and $p+1$ pairwise disjoint $(q-s)$-subsets $T_1,\ldots,T_{p+1}$ of $V(\mathcal H)-Z$ such that $\mathcal S=\{T_1\cup Z,\ldots,T_{p+1}\cup Z\}$; moreover, if so, then $\ext(\mathcal S)=T_1\cup\cdots\cup T_{p+1}$.

Recall that a hypergraph is \emph{hereditary $(p,q)$-Helly} if each of its subhypergraphs is $(p,q)$-Helly. This means that a hypergraph is hereditary $(p,q)$-Helly if and only if each of its $(p,q)$-intersecting partial subhypergraphs has $q^+$-core. Below, we give the aforementioned extension of Theorem~\ref{thm:HpHelly} to hereditary $(p,q)$-Helly hypergraphs. Notice that the equivalence (\ref{it:HpqH1})${}\Leftrightarrow{}$(\ref{it:HpqH7}) below gives a characterization of hereditary $(p,q)$-Helly hypergraphs by minimal forbidden partial subhypergraphs.

\begin{thm}\label{thm:HpqHelly}
If $p$ and $q$ are positive integers, then the following statements are equivalent for each hypergraph $\mathcal H$:
\begin{enumerate}[(i)]
\item\label{it:HpqH1} $\mathcal H$ is hereditary $(p,q)$-Helly;
\item\label{it:HpqH2} $\mathcal H$ is $(p,q')$-Helly for every $q'\geq q$;
\item\label{it:HpqH3} $\mathcal H$ is strong $(p,q)$-Helly;
\item\label{it:HpqH4} every partial $(p+1)$-hypergraph of $\mathcal H$ is strong $(p,q)$-Helly;
\item\label{it:HpqH5} $\varPhi_q(\mathcal H)$ is hereditary $p$-Helly;
\item\label{it:HpqH6} $M(\mathcal H)$ contains no incidence matrix of $\mathcal J_{p+1,q,s}$ as a submatrix for any $s\in\{0,\ldots,q-1\}$;
\item\label{it:HpqH7} no partial subhypergraph of $\mathcal H$ is isomorphic to $\mathcal J_{p+1,q,s}$ for any $s\in\{0,\ldots,q-1\}$;
\item\label{it:HpqH8} for each $s\in\{0,\ldots,q-1\}$, each $((p+1)(q-s)+s)$-subset $U$ of $V(\mathcal H)$, and each $p+1$ pairwise disjoint $(q-s)$-subsets $T_1,\ldots,T_{p+1}$ of $U$ such that each of the sets $U-T_1,\ldots,U-T_{p+1}$ is contained in some edge of $\mathcal H$, the basis $\mathcal S=\{T_1\cup Z,\ldots,T_{p+1}\cup Z\}$, where $Z=U-(T_1\cup\cdots\cup T_{p+1})$, satisfies $\core(\mathcal H_\mathcal S^\cup)\cap(T_1\cup\cdots\cup T_{p+1})\neq\emptyset$.
\item\label{it:HpqH9} for each nontrivial starlike $(p+1,q)$-basis $\mathcal S$ of $\mathcal H$, $\core(\mathcal H^\cup_\mathcal S)\cap\ext(\mathcal S)\neq\emptyset$.
\item\label{it:HpqH10} for each starlike $(p+1,q)$-basis $\mathcal S$ of $\mathcal H$, either $\mathcal H_\mathcal S^\cup$ is empty or $\core(\mathcal H_\mathcal S^\cup)\cap\ext(\mathcal S)\neq\emptyset$.
\end{enumerate}
\end{thm}
\begin{proof} \mbox{(\ref{it:HpqH1})${}\Rightarrow{}$(\ref{it:HpqH2})} Suppose (\ref{it:HpqH2}) does not hold. Thus, there is some $(p,q')$-intersecting partial hypergraph $\mathcal H'$ of $\mathcal H$ having a $q''$-core where $q''<q'$. If $q''<q$, then $\mathcal H$ is not $(p,q)$-Helly by definition and, in particular, not hereditary $(p,q)$-Helly. Hence, we assume, without loss of generality, that $q''\geq q$. If $W$ is a subset of $\core(\mathcal H')$ of cardinality $q''-(q-1)$, then the subhypergraph $\mathcal H''$ of $\mathcal H'$ induced by $V(\mathcal H')-W$ is $(p,q)$-intersecting but has a $(q-1)$-core. As $\mathcal H''$ is a partial subhypergraph of $\mathcal H$, (\ref{it:HpqH1}) does not hold.

\mbox{(\ref{it:HpqH2})${}\Rightarrow{}$(\ref{it:HpqH3})} Suppose (\ref{it:HpqH3}) does not hold. Let $\mathcal H'$ be a nonempty $(p,q)$-intersecting partial hypergraph of $\mathcal H$ such that each nonempty subfamily of $p$ or fewer edges of $\mathcal H'$ has a core properly containing the core of $\mathcal H'$. Let $q'=\vert\core(\mathcal H')\vert$. On the one hand, if $q'<q$, then $\mathcal H'$ is not $(p,q)$-Helly. On the other hand, if $q'\geq q$, then $\mathcal H'$ is $(p,q'+1)$-intersecting but has $q'$-core. In either case, (\ref{it:HpqH2}) does not hold.

\mbox{(\ref{it:HpqH3})${}\Rightarrow{}$(\ref{it:HpqH4})} It follows by definition.

\mbox{(\ref{it:HpqH4})${}\Rightarrow{}$(\ref{it:HpqH5})} Suppose (\ref{it:HpqH4}) holds. Let $\mathcal J$ be any partial $(p+1)$-hypergraph of $\varPhi_q(\mathcal H)$. By the definition of operator $\varPhi_q$, $\mathcal J=\varPhi_q(\mathcal H')$ for some partial $(p+1)$-hypergraph $\mathcal H'$ of $\mathcal H$ such that each edge of $\mathcal H'$ has cardinality at least $q$. We claim that $\mathcal J$ is strong $p$-Helly. As $\mathcal J$ has $p+1$ edges, to prove the claim it suffices to show that there is some partial $p$-hypergraph of $\mathcal J$ whose core equals that of $\mathcal J$. If $\mathcal J$ is not $p$-wise intersecting, then some partial $p$-hypergraph of $\mathcal J$ has empty core, which consequently coincides with the core of $\mathcal J$. Thus, we assume without loss of generality, that $\mathcal J$ is $p$-wise intersecting. Hence, $\mathcal H'$ is $(p,q)$-intersecting because of Remark~\ref{rmk:phiq}. Since (\ref{it:HpqH4}) holds, there is some partial $p$-hypergraph 
$\mathcal H''$ of $\mathcal H'$ such that $\core(\mathcal H'')=\core(\mathcal H')$. Hence, $\varPhi_q(\mathcal H'')$ is a partial $p$-hypergraph of $\mathcal J$ and, by Remark~\ref{rmk:phiq}, $\core(\varPhi_q(\mathcal H''))=\varphi_q(\core(\mathcal H''))=\varphi_q(\core(\mathcal H'))=\core(\varPhi_q(\mathcal H'))=\core(\mathcal J)$. This proves the claim. Therefore, (\ref{it:HpqH5}) holds because of the implication (\ref{it:HpH4})${}\Rightarrow{}$(\ref{it:HpH1}) of Theorem~\ref{thm:HpHelly}.

\mbox{(\ref{it:HpqH5})${}\Rightarrow{}$(\ref{it:HpqH6})} Suppose (\ref{it:HpqH6}) does not hold; i.e., $M(\mathcal H)$ has an incidence matrix of $\mathcal J_{p+1,q,s}$ as a submatrix for some $s\in\{0,\ldots,q-1\}$. Thus, there is some $((p+1)(q-s)+s)$-subset $U$ of $V(\mathcal H)$ and $p+1$ pairwise disjoint $(q-s)$-subsets $T_1,\ldots,T_{p+1}$ of $U$, and edges $E_1,\ldots,E_{p+1}$ of $\mathcal H$ such that $E_i\cap U=U-T_i$ for each $i\in\{1,\ldots,p+1\}$. If $Z=U-(T_1\cup\cdots\cup T_{p+1})$, then it is clear that the incidence matrix of the subhypergraph induced by $\{T_1\cup Z,\ldots,T_{p+1}\cup Z\}$ of the partial hypergraph of $\varPhi_q(\mathcal H)$ formed by the hyperedges $\varphi_q(E_1),\ldots,\varphi_q(E_{p+1})$ is a $(p+1)\times(p+1)$ complement of a permutation matrix. Hence, the implication (\ref{it:HpH5})${}\Rightarrow{}$(\ref{it:HpH1}) of Theorem~\ref{thm:HpHelly} shows that $\varPhi_q(\mathcal H)$ is not hereditary $p$-Helly; i.e., (\ref{it:HpqH5}) does not hold.

\mbox{(\ref{it:HpqH6})${}\Rightarrow{}$(\ref{it:HpqH7})} It follows by definition of incidence matrix.

\mbox{(\ref{it:HpqH7})${}\Rightarrow{}$(\ref{it:HpqH8})} Suppose \eqref{it:HpqH8} does not hold; i.e., there is some $s\in\{0,\ldots,q-1\}$, some $((p+1)(q-s)+s)$-subset $U$ of $V(\mathcal H)$, and $p+1$ pairwise disjoint $(q-s)$-subsets $T_1,\ldots,T_{p+1}$ of $U$ such that each of the sets $U-T_1,\ldots,U-T_{p+1}$ is contained in some edge of $\mathcal H$ and the $(p+1,q)$-basis $\mathcal S=\{T_1\cup Z,\ldots,T_{p+1}\cup Z\}$, where $Z=U-(T_1\cup\cdots\cup T_{p+1})$, satisfies $\core(\mathcal H_\mathcal S^\cup)\cap(T_1\cup\cdots\cup T_{p+1})=\emptyset$. Observe that $\mathcal H_\mathcal S^\cup\neq\emptyset$ because each of the support sets $U-T_1,\ldots,U-T_{p+1}$ of $\mathcal S$ is contained in some edge of $\mathcal H$. Hence, $\core(\mathcal H_\mathcal S^\cup)\cap(T_1\cup\cdots\cup T_{p+1})=\emptyset$ implies that, for each $v\in T_1\cup\cdots\cup T_{p+1}$, we can choose some edge $E_v$ of $\mathcal H_\mathcal S^\cup$ such that $v\notin E_v$. Notice that $E_v$ and $E_{v'}$ may coincide for two different vertices $v$ and $v'$. Moreover, since every edge of $\mathcal H_\mathcal S^\cup$ contains at least $p$ members of $\mathcal S$, necessarily $Z\subseteq E_v$ and $T_j\subseteq E_v$ for each $j\in\{1,\ldots,p+1\}$ such that $v\notin T_j$.

Suppose that there are two vertices $v_1,v_1'\in T_1$ such that $E_{v_1}\cap T_1$ and $E_{v_1'}\cap T_1$ are inclusion-wise incomparable. Thus, there are two vertices $t_1,t_1'\in T_1$ such that $t_1\in E_{v_1'}-E_{v_1}$ and $t_1'\in E_{v_1}-E_{v_1'}$. Therefore, for any choice of a vertex $v_j\in T_j$ for each $j\in\{1,\ldots,p+1\}$, the subhypergraph induced by $\{t_1,t_1',v_2,\ldots,v_p\}\cup(T_{p+1}-\{v_{p+1}\})\cup Z$ of the hypergraph formed by the edges $E_{v_1},E_{v_1'},E_{v_2},\ldots,E_{v_p}$ is a partial subhypergraph of $\mathcal H$ isomorphic to $\mathcal J_{p+1,q,q-1}$ and (\ref{it:HpqH7}) does not hold. Hence, we assume, without loss of generality, that for each $i\in\{1,\ldots,p\}$, the family $\mathcal E_i=\{E_v\cap T_i:\,v\in T_i\}$ is a chain with respect to inclusion and let $E_i$ be the minimum element of $\mathcal E_i$. Since $v\notin E_v$ for each $v\in T_1\cup\cdots\cup T_{p+1}$, $E_i\cap T_i=\emptyset$ for each $i\in\{1,\ldots,p+1\}$. Therefore, for any choice of a vertex $v_i\in T_i$ for each $i\in\{1,\ldots,p+1\}$, the subhypergraph induced by $T_1\cup\cdots\cup T_{p+1}\cup Z$ of the hypergraph formed by the edges $E_1,\ldots,E_{p+1}$ is a partial subhypergraph of $\mathcal H$ isomorphic to $\mathcal J_{p+1,q,s}$ and (\ref{it:HpqH7}) does not hold.

\mbox{(\ref{it:HpqH8})${}\Rightarrow{}$(\ref{it:HpqH9})} Suppose that (\ref{it:HpqH9}) does not hold. Thus, there is some $s\in\{0,\ldots,q-1\}$, some $s$-subset $Z$ of $V(\mathcal H)$, and some pairwise disjoint $(q-s)$-subsets $T_1,\ldots,T_{p+1}$ of $V(\mathcal H)-Z$ such that the basis $\mathcal S=\{T_1\cup Z,\ldots,T_{p+1}\cup Z\}$ satisfies that each of its support sets is contained in some edge of $\mathcal H$ and $\core(\mathcal H_\mathcal S^\cup)\cap(T_1\cup\cdots\cup T_{p+1})=\emptyset$. Let $U=T_1\cup\cdots\cup T_{p+1}\cup Z$. Hence, $U$ is a $((p+1)(q-s)+s)$-subset of $V(\mathcal H)$ and each of the sets $U-T_1,\ldots,U-T_{p+1}$ is a support set of $\mathcal S$ which, by assumption, is contained in some edge of $\mathcal H$. Moreover, since $T_1,\ldots,T_{p+1}$ are $p+1$ pairwise disjoint $(q-s)$-subsets of $U$ and $\core(\mathcal H_\mathcal S^\cup)\cap(T_1\cup\cdots\cup T_{p+1})=\core(\mathcal H^\cup_\mathcal S)\cap\ext(\mathcal S)=\emptyset$, (\ref{it:HpqH8}) does not hold.

\mbox{(\ref{it:HpqH9})${}\Rightarrow{}$(\ref{it:HpqH10})} Suppose~\eqref{it:HpqH9} holds. Let $\mathcal S$ be a trivial starlike $(p+1,q)$-basis of $\mathcal H$. In order to prove that \eqref{it:HpqH10} holds, it suffices to prove that either $\mathcal H^\cup_\mathcal S$ is empty or $\core(\mathcal H^\cup_\mathcal S)\cap\ext(\mathcal S)\neq\emptyset$. By Remark~\ref{rmk:nontrivial}, either $\mathcal H^\cup_\mathcal S$ is empty or $\core(\mathcal H^\cup_\mathcal S)$ contains some member of $\mathcal S$. Notice that if the latter holds, then $\core(\mathcal H^\cup_\mathcal S)\cap\ext(\mathcal S)\neq\emptyset$ because the fact that $\mathcal S$ is starlike implies that each member of $\mathcal S$ contains at least one vertex from $\ext(\mathcal S)$. Therefore, \eqref{it:HpqH10} holds.

\mbox{(\ref{it:HpqH10})${}\Rightarrow{}$(\ref{it:HpqH1})} Suppose~\eqref{it:HpqH1} does not hold. Let $\mathcal H'$ be a $(p,q)$-intersecting partial subhypergraph of $\mathcal H$ having an $s$-core where $s<q$. Let $Z=\core(\mathcal H')$ and $E(\mathcal H')=\{E_1',\ldots,E_{m'}'\}$ where, necessarily, $m'\geq p+1$. We assume, without loss of generality, that $\mathcal H'$ is minimal; i.e., $\mathcal H'-E_i'$ has $q^+$-core for each $i\in\{1,\ldots,m'\}$. Hence, $\core(\mathcal H'-E_i')-Z$ is a $(q-s)^+$-set and let $T_i$ be any $(q-s)$-subset of $\core(\mathcal H'-E_i')-Z$, for each $i\in\{1,\ldots,m'\}$. Since $\mathcal H'$ is a partial subhypergraph of $\mathcal H$, for each $i\in\{1,\ldots,p+1\}$, there is an edge $E_i$ of $\mathcal H$ such that $E_i\cap V(\mathcal H')=E_i'$. By construction, $T_i\cap E_i=\emptyset$ but $T_i\cup Z\subseteq E_j$ for each two different $i,j\in\{1,\ldots,p+1\}$. Thus, $T_1,\ldots,T_{p+1}$ are pairwise disjoint and $\mathcal S=\{T_1\cup Z,\ldots,T_{p+1}\cup Z\}$ is a starlike $(p+1,q)$-basis of $\mathcal H$ such that $E_1,\ldots,E_{p+1}\in E(\mathcal H_\mathcal S^\cup)$. Hence, (\ref{it:HpqH10}) does not hold because $\mathcal H_\mathcal S^\cup$ is nonempty and $\core(\mathcal H_\mathcal S^\cup)\cap\ext(\mathcal S)=\core(\mathcal H_\mathcal S^\cup)\cap(T_1\cup\cdots\cup T_{p+1})\subseteq E_1\cap\cdots\cap E_{p+1}\cap(T_1\cup\cdots\cup T_{p+1})=\emptyset$.\end{proof}

The remaining of this section is devoted to addressing the problem of recognizing hereditary $(p,q)$-Helly hypergraphs. Firstly, we derive from the equivalence (\ref{it:HpqH1})${}\Leftrightarrow{}$(\ref{it:HpqH9}) of the theorem above that Lemma~\ref{lem:technical} is still valid for hereditary $(p,q)$-Helly hypergraphs.

\begin{lem}\label{lem:technical2} Lemma~\ref{lem:technical} is still valid if `$(p,q)$-Helly' is replaced by `hereditary $(p,q)$-Helly'.\end{lem}
\begin{proof} By virtue of the equivalence (\ref{it:HpqH1})${}\Leftrightarrow{}$(\ref{it:HpqH9}) of Theorem~\ref{thm:HpqHelly}, $\mathcal H$ is hereditary $(p,q)$-Helly if and only if $\core(\mathcal H^\cup_\mathcal S)\cap\ext(\mathcal S)\neq\emptyset$ for each nontrivial starlike $(p+1,q)$-basis $\{S_{j_1},\ldots,S_{j_{p+1}}\}$, where $1\leq j_1<\cdots<j_{p+1}\leq N$.

We proceed as in the proof of Lemma~\ref{lem:technical}, except that, each time we visit a node $(j_1,\ldots,j_{k})$ in the second stage, we update $S_{j_1}\cup\cdots\cup S_{j_k}$ (using an $n$-vector to keep track of the number of occurrences of each vertex of $\mathcal H$ among $S_{j_1},\ldots,S_{j_k}$) and update a counter that indicates the number of vertices of $V(\mathcal H)$ that belong to more than one but less than $p+1$ of the sets $S_{j_1},\ldots,S_{j_k}$ in $O(q)$ time, which allows us, at each leaf node $(j_1,\ldots,j_{p+1})$ of the second stage, to decide whether $\mathcal S=\{S_{j_1},\ldots,S_{j_{p+1}}\}$ is starlike in $O(pq)$ time by comparing the value of the counter with $\vert\core(\mathcal S)\vert$. Moreover, since we can compute $\ext(\mathcal S)$ in $O(pq)$ time and the intersection between $\core(\mathcal H^\cup_\mathcal S)$ and $\ext(\mathcal S)$ in $O(pq+r)$ time, all the additional operations take at most $O\bigl((q+pq+r)\frac Np\binom Np\bigr)=O\bigl(rN\binom Np\bigr)$ time, which completes the proof of the lemma.\end{proof}

Using the results above, we derive two recognition algorithms for hereditary $(p,q)$-Helly hypergraphs analogous to those of Theorem~\ref{thm:HpHelly-rec}, both of which are polynomial-time when $p$ and $q$ are fixed.

\begin{thm}\label{thm:HpqHelly-rec} If $q$ is any fixed positive integer, then the recognition problem for hereditary $(p,q)$-Helly hypergraphs, where $p$ is part of the input:
\begin{enumerate}[(i)]
 \item\label{it:HpqHrec-1} can be solved in $O\bigl(p^2r\binom m{p+1}\bigr)$ time;
 \item\label{it:HpqHrec-2} can be solved in $O\bigl((rN+M)\binom Np\bigr)$ time, where $N=\binom nq$.
\end{enumerate}
In particular, if $p$ and $q$ are both fixed, the above bounds become $O(rm^{p+1})$ and $O(rn^{(p+1)q}+Mn^{pq})$, respectively.
\end{thm}
\begin{proof} Let $\mathcal H$ be the input hypergraph. Bound (\ref{it:HpqHrec-1}) follows from the equivalence (\ref{it:HpqH1})${}\Leftrightarrow{}$(\ref{it:HpqH4}) of Theorem~\ref{thm:HpqHelly}. In fact, for each partial $(p+1)$-hypergraph $\mathcal H'$ of $\mathcal H$, we can compute the core of $\mathcal H'$ and every partial $p$-hypergraph of $\mathcal H'$ in $O(p^2r)$ time. Since there are $O\bigl(\binom m{p+1}\bigr)$ such partial hypergraphs $\mathcal H'$, deciding whether statement (\ref{it:HpqH4}) of Theorem~\ref{thm:HpqHelly} holds takes at most $O\bigl(p^2r\binom m{p+1}\bigr)$ time. The derivation of bound (\ref{it:HpqHrec-2}) from Lemma~\ref{lem:technical2} is analogous to that of bound (\ref{it:2}) of Theorem~\ref{thm:rec-pqHelly} from Lemma~\ref{lem:technical}.\end{proof}

We can also extend the hardness result contained in (\ref{it:HpH-rec3}) of Theorem~\ref{thm:HpHelly-rec} to hereditary $(p,q)$-Helly hypergraphs as follows.

\begin{thm} The recognition problem for hereditary $(p,q)$-Helly hypergraphs, for positive integers $p$ and $q$, is co-NP-complete if $p$ is part of the input (even if $q$ is fixed).\end{thm}
\begin{proof} Let $\mathcal H$ be the input hypergraph. The recognition problem is in co-NP because, by the equivalence (\ref{it:HpqH1})${}\Leftrightarrow{}$(\ref{it:HpqH4}) of Theorem~\ref{thm:HpqHelly}, if $\mathcal H$ is not hereditary $(p,q)$-Helly, then there is a certificate in the form of a partial $(p+1)$-hypergraph of $\mathcal H$ which is not $(p,q)$-strong Helly. Assume $q\geq 1$ fixed and let $\mathcal H'$ be the hypergraph that arises from $\mathcal H$ by adding $q-1$ new vertices $v_1,\ldots,v_{q-1}$ to its vertex set as well as to each of its edges. Clearly, $M(\mathcal H')$ arises from $M(\mathcal H)$ by adding $q-1$ columns filled with $1$'s. Hence, (\ref{it:HpqH1})${}\Leftrightarrow{}$(\ref{it:HpqH6}) of Theorem~\ref{thm:HpqHelly} implies that $\mathcal H$ is hereditary $p$-Helly if and only if the hypergraph $\mathcal H'$ is hereditary $(p,q)$-Helly. The result now follows from statement (\ref{it:HpH-rec3}) of Theorem~\ref{thm:HpHelly-rec}.\end{proof}

We prove the following result for future reference.

\begin{lem}\label{lem:H1qHelly} If $q$ is a positive integer, then the following statements are equivalent for each simple hypergraph $\mathcal H$:
\begin{enumerate}[(i)]
 \item\label{it:H1qH1} $\mathcal H$ is hereditary $(1,q)$-Helly
 \item\label{it:H1qH2} $\mathcal H$ has at most one $q^+$-edge.
 \item\label{it:H1qH3} The union of all the $q$-subsets of the edges of $\mathcal H$ is empty or is contained in some edge of $\mathcal H$.
 \item\label{it:H1qH4} The union of all the $q$-subsets of the edges of $\mathcal H$ is empty or an edge of $\mathcal H$.
\end{enumerate}\end{lem}
\begin{proof} By the equivalence (\ref{it:HpqH1})${}\Leftrightarrow{}$(\ref{it:HpqH3}) of Theorem~\ref{thm:HpqHelly}, $\mathcal H$ is hereditary $(1,q)$-Helly if and only if, for every two $q^+$-edges $E_1$ and $E_2$ of $\mathcal H$, either $E_1\cap E_2=E_1$ or $E_1\cap E_2=E_2$, which in turn holds precisely when $E_1=E_2$ (because $\mathcal H$ is simple). This proves (\ref{it:H1qH1})${}\Leftrightarrow{}$(\ref{it:H1qH2}). Let $U$ denote the union of all the $q$-subsets of the edges of $\mathcal H$. The implication (\ref{it:H1qH2})${}\Rightarrow{}$(\ref{it:H1qH4}) is clear because, if $\mathcal H$ has at most one $q^+$-edge, then either $\mathcal H$ has no $q^+$-edges and $U=\emptyset$, or $\mathcal H$ has exactly one $q^+$-edge and $U=E$. Since (\ref{it:H1qH4})${}\Rightarrow{}$(\ref{it:H1qH3}) holds trivially, it only remains to show that (\ref{it:H1qH3})${}\Rightarrow{}$(\ref{it:H1qH2}). In order to do so, suppose that (\ref{it:H1qH2}) does not hold and let $E_1$ and $E_2$ be two different $q^+$-edges of $\mathcal H$. Thus, $U$ contains $E_1\cup E_2$ and, since $\mathcal H$ is simple, $E_1$ and $E_2$ are inclusion-wise incomparable. In particular, $U$ properly contains $E_1$. Hence, (\ref{it:H1qH3}) does not hold, since otherwise $U$ would be contained in some edge $E_3$ of $\mathcal H$ which would properly contain the edge $E_1$. This completes the proof of (\ref{it:H1qH3})${}\Rightarrow{}$(\ref{it:H1qH2}) and thus of the lemma.\end{proof}

\section{The $(p,q)$-clique-Helly property of graphs}\label{sec:HpqCHelly}

In this section, we study the problems of characterizing and recognizing $(p,q)$-clique-Helly graphs and hereditary $(p,q)$-clique-Helly graphs. We generalize the characterization of $p$-clique-Helly graphs proved in \cite{MR2365055} in terms of expansions to $(p,q)$-clique-Helly graphs and make an improvement in the time complexity given in the same work for the recognition problem of $(p,q)$-clique-Helly graphs. We also characterize hereditary $(p,q)$-clique graphs in several ways, including a characterization by forbidden induced subgraphs, and derive a polynomial-time recognition algorithm for each fixed $p$ and $q$. In contrast, we show that the recognition problem is NP-hard if $p$ or $q$ is part of the input. Our results for hereditary $(p,q)$-clique-Helly graphs generalize results for hereditary $p$-clique-Helly graphs proved in \cite{MR2404220}.

\subsection{$(p,q)$-clique-Helly graphs}

For each positive integer $q$, let $\varPhi_q(G)$ be the graph defined as follows: the vertices of $\varPhi_q(G)$ are the $q$-cliques of $G$ and two different vertices of $\varPhi_q(G)$ are adjacent if and only if they are contained in a common clique of $G$. The operator $\varPhi_q$ coincides with the operator $\varPhi_{q,2q}$ defined in \cite[p.\ 136]{Marcia-doc} and $\varPhi_2$ is the \emph{edge clique graph operator} introduced in~\cite{MR753607}. The operator $\varPhi_q$ was used in~\cite{MR2365055} in order to characterize $(p,q)$-clique-Helly graphs as follows.

\begin{thm}[{\cite{MR2365055}}]\label{thm:Phi}
A graph $G$ is $(p,q)$-clique-Helly if and only if $\varPhi_q(G)$ is $p$-clique-Helly.
\end{thm}

Given a graph $G$, the \emph{clique hypergraph} $\mathcal C(G)$ of $G$ is the hypergraph whose vertex set coincides with the vertex set of $G$ and whose edge family is the set of maximal cliques of $G$. By definition, a graph is $p$-clique-Helly if and only if $\mathcal C(G)$ is $p$-Helly. The connection between the operator $\varPhi_q$ for graphs and the operator $\varPhi_q$ for hypergraphs defined in the preceding section is made explicit by the result below. The restriction to $q=2$ of the theorem below also appeared in \cite{MR753607} and \cite{Marcia-doc}.

\begin{thm}[Clique Preservation Property~\cite{MR2365055}]\label{thm:CPP} For each positive integer $q$ and each graph $G$, $\mathcal C(\varPhi_q(G))=\varPhi_q(\mathcal C(G))$; i.e., the maximal cliques of $\varPhi_q(G)$ are precisely the results of applying $\varphi_q$ to each of the maximal cliques of $G$ of cardinality at least $q$.\end{thm}

The above result shows that Theorem~\ref{thm:Phi} is a specialization of Corollary~\ref{cor:var2} to the case $\mathcal H=\mathcal C(G)$.  Interestingly, the specialization of Corollary~\ref{cor:pqHelly} to the case $\mathcal H=\mathcal C(G)$ directly leads to a generalization to $(p,q)$-clique-Helly graphs of the characterization of $p$-clique-Helly graphs given in~\cite{MR2365055} in terms of expansions. If $Q$ is a $(p+1)$-clique of a graph $G$, the \emph{$(p+1)$-expansion} of $Q$ in $G$ is the subgraph of $G$ induced by those vertices of $G$ which are adjacent to at least $p$ vertices of $Q$. This notation was introduced in~\cite{MR2365055}. The $3$-expansions were originally used in~\cite{Dra-thesis} and \cite{MR1447758} to characterize clique-Helly graphs. The characterization of $p$-clique-Helly given in terms of expansions is as follows.

\begin{thm}[{\cite{MR2365055}}]\label{thm:(p+1)-exp} For each integer $p\geq 2$, a graph $G$ is $p$-clique-Helly if and only if every $(p+1)$-expansion in $G$ contains a universal vertex.\end{thm}

Before generalizing the above result, we prove the lemma below. Recall that if $\mathcal H$ is any hypergraph and $P$ is a subset of $V(\mathcal H)$, then $\mathcal H_P$ denotes the partial hypergraph of $\mathcal H$ whose edge family consists of those edges of $\mathcal H$ containing $P$.

\begin{lem}\label{lem:CP} Let $G$ be a graph and let $P$ be a clique of $G$. If $\mathcal C$ is the clique hypergraph of $G$, then $\core(\mathcal C_P)$ is the set of universal vertices of $G[V(\mathcal C_P)]$ or, equivalently, $\core(\mathcal C_P)$ is the set of vertices of $G$ that are complete to the set of vertices of $G$ that are complete to $P$. In particular, $\core(\mathcal C_P)$ is a clique of $G$ containing $P$ and can be computed in $O(m+n)$ time.\end{lem}
\begin{proof} For each subset $W$ of $V(G)$, let $\comp_G(W)$ denote the set of vertices of $G$ that are complete to $W$. Let $v\in V(G)$. By definition, $v\notin\core(\mathcal C_P)$ if and only if there is some maximal clique $Q$ of $G$ containing $P$ such that $v\notin Q$. Equivalently, $v\notin\core(\mathcal C_P)$ if and only if $v$ is nonadjacent to some vertex $w$ which is complete to $P$. This proves that $\core(\mathcal C_P)=\comp_G(\comp_G(P))$. Since $\mathcal C_P$ consists of the maximal cliques of $G$ containing $P$, $V(\mathcal C_P)=\comp_G(P)$ and, since $P$ is a clique, $\comp_G(\comp_G(P))\subseteq\comp_G(P)$. Therefore, $\core(\mathcal C_P)$ is also equivalent to the set of universal vertices of $G[V(\mathcal C_P)]$. In particular, $\core(\mathcal C_P)$ is a clique of $G$ containing $P$. Moreover, since clearly $\comp_G(W)$ can be computed in $O(m+n)$ time, $\core(\mathcal C_P)=\comp_G(\comp_G(P))$ can also be computed in $O(m+n)$ time.\end{proof}

We introduce $(p+1,q)$-expansions as follows. Let $G$ be a graph and let $\mathcal Q$ be a family of $p+1$ pairwise different $q$-cliques of $G$ such that each union of $p$ members of $\mathcal Q$ is a clique of $G$. Observe that this is equivalent to $\mathcal Q$ being a nontrivial $(p+1,q)$-basis of $\mathcal C(G)$ and that if $p\geq 2$ then this means that all the members of $\mathcal Q$ are contained in a common clique of $G$. We define the \emph{$(p+1,q)$-expansion} of $\mathcal Q$ in $G$ as the subgraph of $G$ induced by those vertices of $G$ which are complete to at least $p$ members of $\mathcal Q$. The characterization of $(p,q)$-clique-Helly graphs in terms of expansions is the following.

\begin{thm}\label{thm:pqcliqueHelly} For any positive integers $p$ and $q$, a graph $G$ is $(p,q)$-clique-Helly if and only if each $(p+1,q)$-expansion in $G$ has at least $q$ universal vertices.\end{thm}
\begin{proof} Let $\mathcal C$ denote the clique hypergraph of $G$. By definition, the $(p+1,q)$-expansions in $G$ are the induced subgraphs $G[V(\mathcal C_\mathcal Q^\cup)]$ as $\mathcal Q$ varies over the nontrivial $(p+1,q)$-bases of $\mathcal C$. Since, also by definition, $G$ is $(p,q)$-clique-Helly if and only if $\mathcal C$ is $(p,q)$-Helly, Corollary~\ref{cor:pqHelly} implies that in order to prove the theorem it is enough to show that $\core(\mathcal C_\mathcal Q^\cup)$ coincides with the set of universal vertices of $G[V(\mathcal C_\mathcal Q^\cup)]$.

Let $\mathcal Q$ be a nontrivial $(p+1,q)$-basis of $\mathcal C$. Thus, the support sets $Q_1,\ldots,Q_{p+1}$ of $\mathcal Q$ are cliques of $G$ and, by Lemma~\ref{lem:CP}, $\core(\mathcal C_{Q_i})$ is the set of universal vertices of $G[V(\mathcal C_{Q_i})]$. Therefore, $\core(\mathcal C_\mathcal Q^\cup)=\core(\mathcal C_{Q_1})\cap\cdots\cap\core(\mathcal C_{Q_{p+1}})$ is the set of universal vertices of $G[V(\mathcal C_{Q_1})\cup\cdots\cup V(\mathcal C_{Q_{p+1}})]=G[V(\mathcal C_{\mathcal Q}^\cup)]$, as needed.\end{proof}

We now turn to the problem of recognizing $(p,q)$-clique-Helly graphs for any positive integers $p$ and $q$. If $p$ or $q$ is part of the input, the problem is known to be NP-hard~\cite{MR2365055}. Nevertheless, as a consequence of Theorem~\ref{thm:Phi}, a polynomial-time recognition algorithm for $(p,q)$-clique-Helly graphs for each fixed $p$ and $q$ was proposed in~\cite{MR2365055} (see also \cite{D-P-S-ds} for more details).

\begin{thm}[\cite{MR2365055}] The recognition problem for $(p,q)$-clique-Helly graphs:
\begin{enumerate}[(i)]
 \item\label{it:time} can be solved in $O(n^{(p+3)q})$ time for each fixed $q$, where $p$ is part of the input;
 \item is NP-hard if $p$ or $q$ is part of the input.
\end{enumerate}\end{thm}

We will derive an improvement upon the time complexity bound (\ref{it:time}) above. Our algorithm uses a refinement of the precomputing approach of~\cite{MR2457935}. Moreover, our algorithm matches the $O(m^2+n)$ time bound proved in~\cite{MR2321859} for recognizing clique-Helly graphs, whereas we obtain an $O(\omega m^2+n)$ time bound for the recognition of $(2,2)$-clique-Helly graphs, which represents an improvement upon the $O(m^5+n)$ time bound for the same problem proved in \cite{Marcia-doc}. We rely on the following algorithmic result about the enumeration of all the $q$-cliques of a graph.

\begin{thm}[\cite{MR774940}]\label{thm:qcliques} Given a connected graph $G$ and an integer $q\geq 2$, the $q$-cliques of $G$ can be enumerated in $O(qm^{q/2})$ time. In particular, the number of $q$-cliques of $G$ is $O(m^{q/2})$. \end{thm}

Our time bounds for the recognition of $(p,q)$-clique-Helly graphs are as follows.

\begin{thm}\label{thm:recog-pq-CH} If $q$ is any fixed positive integer, then recognition problem for $(p,q)$-clique-Helly graphs, where $p$ is part of the input, can be solved in:
\begin{enumerate}[(i)]
 \item $O(m+n)$ time if $p=1$ and $q=1$.
 \item $O(m^{q/2+1}+n)$ time if $p=1$ and $q\geq 2$;
 \item\label{it:pqCH-rec3} $O(m^{p/2+1}+p\omega m^{(p+1)/2}+n)$ time if $p\geq 2$ and $q=1$;
 \item\label{it:pqCH-rec4} $O\bigl(qm^{q/2}+\omega m^{q/2}\binom{N_q}p+n\bigr)$ time if $p\geq 2$ and $q\geq 2$, where $N_q$ is the number of $q$-cliques of the input graph.
\end{enumerate}
In particular, if $p$ and $q$ are both fixed, bound (\ref{it:pqCH-rec4}) above becomes $O(\omega m^{(p+1)q/2}+n)$.\end{thm}
\begin{proof} Let $G$ be the input graph and let $\mathcal C$ be the clique hypergraph of $G$.

Suppose first that $p=1$. If $q=1$ then, by the equivalence (\ref{it:1qH1})${}\Leftrightarrow{}$(\ref{it:1qH2}) of Remark~\ref{rmk:1qHelly}, $G$ is $(p,q)$-Helly if and only if $G$ has a universal vertex, which can be decided in $O(m+n)$ time. Suppose now that $q\geq 2$. We assume that $2m\geq q$, since if the opposite is true it can be detected in $O(m)$ time and, if so, $G$ has no $q$-clique and thus is $(p,q)$-clique-Helly. The equivalence (\ref{it:1qH1})${}\Leftrightarrow{}$(\ref{it:1qH3}) of Remark~\ref{rmk:1qHelly} implies that $G$ is $(p,q)$-Helly if and only if there are at least $q$ vertices in the intersection of the cores of $\mathcal C_Q$ for all the $q$-cliques $Q$ of $G$. Since we can compute the components of $G$ in $O(m+n)$ time, the cores of $\mathcal C_Q$ for all the $q$-cliques $Q$ of every component of $G$ in $O(m^{q/2+1})$ time (by Lemma~\ref{lem:CP} and Theorem~\ref{thm:qcliques}), and the intersection of all such cores in $O(\omega m^{q/2}+n)$ time, we can decide whether $G$ is $(p,q)$-clique-Helly in $O(m^{q/2+1}+n)$ time.

We assume, from now on, that $p\geq 2$. Since $G$ is $(p,q)$-clique-Helly if and only if each component of $G$ is $(p,q)$-clique-Helly and the components of $G$ can be determined in $O(m+n)$ time, we assume, without loss of generality, that $G$ is connected and thus $O(m+n)=O(m)$.

We consider first the case $q=1$. We start by computing all the $p$-cliques and all the $(p+1)$-cliques of $G$ in $O(pm^{(p+1)/2})$ time (by Theorem~\ref{thm:qcliques}). By Corollary~\ref{cor:pqHelly} applied to $\mathcal H=\mathcal C$, it follows that $G$ is $(p,1)$-clique-Helly if and only if $\mathcal C^\cup_P$ has nonempty core for each $(p+1)$-clique $P$ of $G$. Thus, if $G$ has no $(p+1)$-clique, then $G$ is trivially $(p,1)$-clique-Helly. Hence, we assume, without loss of generality, that $p+1\leq\omega$. As $p\geq 2$, necessarily $m\geq 3$. We assume some fixed total order on $V(G)$. In additional $O(p^2m^{p/2})$ time, we can assume that each $p$-clique $P=\{v_1,\ldots,v_p\}$ of $G$, where $v_1<\cdots<v_p$, is stored as a sorted sequence $v_1,\ldots,v_p$. Then, in additional $O(p^2m^{p/2})$ time, we can build a trie $T_1$ containing all the $p$-cliques of $G$, where the children of each node of $T_1$ are stored in a doubly linked list sorted by their largest vertex (in the fixed total order), in such a way that, if $P$ is any clique $\{v_1,\ldots,v_k\}$ contained in some $p$-clique $\{v_1,\ldots,v_p\}$ of $G$, where $v_1<\cdots<v_p$, then $P$ is represented by the node of $T_1$ corresponding to the ``word'' formed by the sequence of $k$ ``letters'' $v_1,\ldots,v_k$. In particular, the leaf nodes of $T_1$ are precisely those representing the $p$-cliques of $G$. For each leaf node of $T_1$, representing some $p$-clique $P$, we use Lemma~\ref{lem:CP} to compute $\core(\mathcal C_P)$ in $O(m)$ time and attach this core to the leaf node. Since $T_1$ has $O(m^{p/2})$ leaf nodes, we conclude that building $T_1$, including the cores attached to its leaf nodes, takes $O((p^2+m)m^{p/2})$ time. Afterwards, we analogously build a second trie $T_2$ containing all the $(p+1)$-cliques $P$ of $G$, except that this time we do not compute the cores attached to its leaf nodes, in $O(p^2m^{(p+1)/2})$ time. We perform a depth-first search on $T_2$ while keeping at all times $p+2$ iterators $I_0,I_1,\ldots,I_{p+1}$. Each time we visit a node of $T_2$ representing some clique $\{v_1,\ldots,v_k\}$ where $v_1<\cdots<v_k$ and $0\leq k\leq p+1$, we update the iterators in such a way that iterator $I_0$ points to the node of $T_1$ representing $P=\{v_1,\ldots,v_{\min\{k,p\}}\}$, while iterator $I_\ell$ points to the node of $T_1$ representing the clique $\{v_1,\ldots,v_{\ell-1},v_{\ell+1},\ldots,v_k\}$ for each $\ell\in\{1,\ldots,k\}$. Updating the iterators takes $O(p)$ time per node (reasoning as in the proof of Lemma~\ref{lem:technical}). Hence, each time we visit a leaf node of $T_2$, representing some $p$-clique $P=\{v_1,\ldots,v_{p+1}\}$, we can compute the core of $\mathcal C^\cup_P$ as the intersection of the cores of the leaf nodes of $T_1$ pointed by $I_1,\ldots,I_{p+1}$ in $O(p\omega)$ time. As discussed at the beginning of the paragraph, $G$ is $(p,q)$-clique-Helly if and only if all such cores of $\mathcal C_P^\cup$ are nonempty.  Hence, since $T_2$ has $O(m^{(p+1)/2})$ leaf nodes, we can decide whether $G$ is $(p,q)$-clique Helly in $O((p^2+m)m^{p/2}+p^2m^{(p+1)/2}+p\omega m^{(p+1)/2})=O(m^{p/2+1}+p\omega m^{(p+1)/2})$ time. Adding the $O(m+n)$ time required to compute the components of $G$, we obtain bound~\eqref{it:pqCH-rec3}.

It only remains to consider the case $p\geq 2$ and $q\geq 2$. Again, we assume, without loss of generality, that $2m\geq q$. We can compute all the $q$-cliques $S_1,\ldots,S_{N_q}$ of $G$ in $O(qm^{q/2})$ time (by Theorem~\ref{thm:qcliques}). Moreover, by Lemma~\ref{lem:CP}, given any subset $P$ of $V(G)$, we can compute $\core(\mathcal C_P)$ in $O(m)$ time. Hence, as $S_1,\ldots,S_{N_q}$ includes all possible $q$-subsets of the edges of $\mathcal C$, Lemma~\ref{lem:technical} implies that it can be decided whether $G$ is $(p,q)$-clique-Helly in $O\bigl(qm^{q/2}+\omega m^{q/2}\binom{N_q}p\bigr)$ time. Bound~\eqref{it:pqCH-rec4} arises by taking into account the $O(m+n)$ time required to find the components of $G$.\end{proof}

\subsection{Hereditary $(p,q)$-clique-Helly graphs}

The remaining of this section is devoted to the problems of characterizing and recognizing hereditary $(p,q)$-clique-Helly graphs. We begin by revisiting the existing results on hereditary $p$-clique-Helly graphs. By definition, $G$ is hereditary $p$-clique-Helly if and only if $\mathcal C(G')$ is $p$-Helly for each induced subgraph $G'$ of $G$. It is well known that if $G'$ is an induced subgraph of $G$, then $\mathcal C(G')$ is the hypergraph formed by the inclusion-wise maximal edges of the subhypergraph of $\mathcal C(G)$ induced by $V(G')$. Hence, $\mathcal C(G)$ may have partial subhypergraphs that are not clique hypergraphs of any induced subgraph of $G$. A graph $G$ is \emph{strong $p$-clique-Helly} if $\mathcal C(G)$ is strong $p$-Helly.

Prisner~\cite{MR1238872} characterized hereditary $2$-clique-Helly graphs in several ways, including a characterization by forbidden induced subgraphs. A characterization of $p$-clique-Helly graphs by forbidden induced subgraphs for every $p\geq 2$ was given in~\cite{MR2404220} in terms of $(p+1)$-oculars. For every integer $p\geq 2$, a graph is a \emph{$(p+1)$-ocular}~\cite{MR2404220} if its vertex set is the union of two disjoint $(p+1)$-sets $U=\{u_1,u_2,\ldots,u_{p+1}\}$ and $W=\{w_1,w_2,\ldots,w_{p+1}\}$, where $U$ is a clique and the only nonneighbor of $w_i$ in $U$ is $u_i$ for each $i\in\{1,\ldots,p+1\}$. Notice that there are no restrictions on the subgraph of $G$ induced by $W$. For each $p\geq 2$, a $p$-clique $Q'$ contained in a $(p+1)$-clique $Q$ is \emph{good} if each vertex complete to $Q'$ is complete to $Q$.

The following characterizations of hereditary $p$-clique-Helly graph were given in \cite{MR2404220}.

\begin{thm}[\cite{MR2404220}]\label{thm:HpCHelly}
For each integer $p\geq 2$, the following statements are equivalent for each graph $G$:
\begin{enumerate}[(i)]
\item\label{it:HpCH1} $G$ is hereditary $p$-clique-Helly;
\item\label{it:HpCH2} $G$ is strong $p$-clique-Helly;
\item\label{it:HpCH3} $G$ contains no induced $(p+1)$-ocular.
\item\label{it:HpCH4} Every $(p+1)$-clique of $G$ contains a good $p$-clique.
\end{enumerate}
\end{thm}

Moreover, the following algorithmic consequences of the above theorem were also proved in \cite{MR2404220}.

\begin{thm}[\cite{MR2404220}]\label{thm:HpCHelly-rec} The recognition problem for hereditary $p$-clique-Helly graphs, where $p$ is part of the input:
\begin{enumerate}[(i)]
 \item\label{it:HpCH-rec1} can be solved in $O(pn^{p+2})$ time;
 \item\label{it:HpCH-rec2} is NP-hard.
\end{enumerate}\end{thm}

For the class of hereditary $2$-clique-Helly graphs, an $O(n^2m)$-time recognition algorithm was proposed in~\cite{MR1238872}. Later, a faster $O(m^2+n)$-time algorithm was devised in \cite{MR2321859}.

Our aim is to extend the above results to hereditary $(p,q)$-clique-Helly graphs for any positive integers $p$ and $q$. For that purpose, we generalize $(p+1)$-oculars as follows. If $p$ and $q$ are positive integers and $s\in\{0,\ldots,q-1\}$, a \emph{$(p+1,q,s)$-ocular} is a graph whose vertex set is the union of two disjoint sets $U$ and $W$ where $U$ is a $((p+1)(q-s)+s)$-set, $T_1,\ldots,T_{p+1}$ are $p+1$ pairwise disjoint $(q-s)$-subsets of $U$, and one of the following statements holds:
\begin{enumerate}[({$\alpha$}$_1$)]
 \item\label{it:oc-a} $p=1$, $W=\emptyset$, and $U-T_i$ is a clique but $(U-T_i)\cup\{v_i\}$ is not a clique for each $v_i\in T_i$ and each $i\in\{1,2\}$;
 \item\label{it:oc-b} $p\geq 2$, $W=\{w_1,\ldots,w_{p+1}\}$, $U$ is a clique, and $w_i$ is complete to $U-T_i$ and anticomplete to $T_i$ for each $i\in\{1,\ldots,p+1\}$.
\end{enumerate}
Observe that if $p\geq 2$ then the vertices of $W$ may induce an arbitrary graph. Clearly, the notion of $(p+1)$-oculars coincide with that of $(p+1,1,0)$-oculars for each $p\geq 2$. We define the \emph{$2$-ocular} as the graph $2K_1$ (which is the only $(2,1,0)$-ocular). It is easy to see that, with this definition, statements (\ref{it:HpCH1}), (\ref{it:HpCH2}), and (\ref{it:HpCH3}) of Theorem~\ref{thm:HpCHelly} are still equivalent for $p=1$.

The lemma below shows that the $(p+1,q,s)$-oculars are not even $(p,q)$-clique-Helly graphs.

\begin{lem}\label{lem:pqs-ocular} If $p$ and $q$ are positive integers and $s\in\{0,\ldots,q-1\}$, then no $(p+1,q,s)$-ocular is $(p,q)$-clique-Helly.\end{lem}
\begin{proof} Let $G$ be a $(p+1,q,s)$-ocular. If $p=1$, then $G$ is not $(p,q)$-Helly because $U-T_1$ and $U-T_2$ are maximal cliques of cardinality $q$ each, whose intersection has cardinality $s$ where $s<q$. Hence, we assume, without loss of generality, that $p\geq 2$. By definition, $Q_i=(U-T_i)\cup\{w_i\}$ is a maximal clique of $G$ for each $i\in\{1,\ldots,p+1\}$ and let $\mathcal Q=\{Q_1,\ldots,Q_{p+1}\}$. Clearly, the core of $\mathcal Q$ is the $s$-set $Z=U-(T_1\cup\cdots\cup T_{p+1})$ while the core of $\mathcal Q-\{Q_i\}$ is the $q$-set $Z\cup T_i$ for each $i\in\{1,\ldots,p+1\}$. Since $s<q$, this proves that $G$ is not $(p,q)$-clique-Helly.\end{proof}

A \emph{clique-matrix} $C(G)$ of $G$ is an incidence matrix of $\mathcal C(G)$; i.e., $C(G)$ is a $(0,1)$-matrix having one row for each maximal clique of $G$, one column for each vertex of $G$, and having a $1$ in the intersection of a row and a column if the corresponding maximal clique contains the corresponding vertex. Clearly, the clique-matrix $C(G)$ of $G$ is unique up to permutation of its rows and/or columns. We say a graph $G$ is \emph{strong $(p,q)$-clique-Helly} if $\mathcal C(G)$ is strong $(p,q)$-clique-Helly.

The theorem below is the main result of this section. It characterizes hereditary $(p,q)$-clique-Helly graphs in several ways, including a characterization by forbidden induced subgraphs (equivalence (\ref{it:HpqCH1})${}\Leftrightarrow{}$(\ref{it:HpqCH8}) below).

\begin{thm}\label{thm:HpqCHelly}
If $p$ and $q$ are positive integers, then the following statements are equivalent for each graph $G$:
\begin{enumerate}[(i)]
\item\label{it:HpqCH1} $G$ is hereditary $(p,q)$-clique-Helly;
\item\label{it:HpqCH2} $G$ is $(p,q')$-clique-Helly, for every $q'\geq q$;
\item\label{it:HpqCH3} $G$ is strong $(p,q)$-clique-Helly;
\item\label{it:HpqCH4} Every family of $p+1$ maximal cliques of $G$ is strong $(p,q)$-Helly;
\item\label{it:HpqCH5} $\varPhi_q(G)$ is hereditary $p$-clique-Helly;
\item\label{it:HpqCH6} $C(G)$ contains no incidence matrix of $\mathcal J_{p+1,q,s}$ as a submatrix for any $s\in\{0,\ldots,q-1\}$;
\item\label{it:HpqCH7} for each $s\in\{0,\ldots,q-1\}$, each $((p+1)(q-s)+s)$-subset $U$ of $V(G)$, and each $p+1$ pairwise disjoint $(q-s)$-subsets $T_1,\ldots,T_{p+1}$ of $U$ such that $U-T_1,\ldots,U-T_{p+1}$ are cliques of $G$, there is some $i\in\{1,\ldots,p+1\}$ and some $v\in T_i$ such that $v$ is adjacent to every vertex of $G$ that is complete to $U-T_i$.
\item\label{it:HpqCH8} $G$ contains no induced $(p+1,q,s)$-ocular for any $s\in\{0,\ldots,q-1\}$.
\end{enumerate}
\end{thm}
\begin{proof} We claim that (\ref{it:HpqCH5}) is equivalent to statement (\ref{it:HpqH5}) of Theorem~\ref{thm:HpqHelly} for $\mathcal H=\mathcal C(G)$. Indeed, Theorem~\ref{thm:HpCHelly} implies that $\varPhi_q(G)$ is hereditary $p$-clique-Helly if and only if $\varPhi_q(G)$ is strong $p$-clique-Helly. By definition, the latter holds if and only if $\mathcal C(\varPhi_q(G))$ is strong $p$-Helly. Hence, by Theorem~\ref{thm:CPP}, $\varPhi_q(G)$ is hereditary $p$-clique-Helly if and only if $\varPhi_q(\mathcal C(G))$ is strong $p$-Helly or, equivalently, hereditary $p$-Helly (by Theorem~\ref{thm:HpHelly}). This proves the claim.

Our second claim is that (\ref{it:HpqCH7}) is equivalent to statement (\ref{it:HpqH8}) of Theorem~\ref{thm:HpqHelly} for $\mathcal H=\mathcal C(G)$. Let us denote $\mathcal C(G)$ simply by $\mathcal C$ and let $s\in\{0,\ldots,q-1\}$, let $U$ be a $((p+1)(q-s)+s)$-subset of $V(G)$, let $T_1,\ldots,T_{p+1}$ be pairwise disjoint $(q-s)$-subsets of $U$ such that each of $U-T_1,\ldots,U-T_{p+1}$ is a clique of $G$, and let $\mathcal S=\{T_1\cup Z,\ldots,T_{p+1}\cup Z\}$ where $Z=U-(T_1\cup\cdots\cup T_{p+1})$. Since the support sets of $\mathcal S$ are $U-T_1,\ldots,U-T_{p+1}$, we have that $\core(\mathcal C_\mathcal S^\cup)=\core(\mathcal C_{U-T_1})\cap\cdots\cap\core(\mathcal C_{U-T_{p+1}})$. Moreover, since $U-T_j$ is a clique of $G$, Lemma~\ref{lem:CP} implies that $U-T_j\subseteq\core(\mathcal C_{U-T_j})$ and, in particular, $T_i\subseteq\core(\mathcal C_{U-T_j})$, for every two different $i,j\in\{1,\ldots,p+1\}$. Hence, $\core(\mathcal C_\mathcal S^\cup)\cap(T_1\cup\cdots\cup T_{p+1})\neq\emptyset$ if and only if there is some $i\in\{1,\ldots,p+1\}$ and some $v\in T_i$ such that $v\in\mathcal\core(\mathcal C_{U-T_i})$ which, by Lemma~\ref{lem:CP}, is equivalent to the fact that $v$ is adjacent to every vertex that is complete to $U-T_i$. This proves our second claim.

From the above arguments, we conclude that Theorem~\ref{thm:HpqHelly} applied to $\mathcal H=\mathcal C(G)$ implies that statements (\ref{it:HpqCH2}) to (\ref{it:HpqCH7}) are equivalent. Hence, in order to prove that (\ref{it:HpqCH1}) and (\ref{it:HpqCH8}) are also equivalent to all of them, it is enough to prove (\ref{it:HpqCH3})${}\Rightarrow{}$(\ref{it:HpqCH1}), (\ref{it:HpqCH1})${}\Rightarrow{}$(\ref{it:HpqCH8}), and (\ref{it:HpqCH8})${}\Rightarrow{}$(\ref{it:HpqCH7}), as we do below.

\mbox{(\ref{it:HpqCH3})${}\Rightarrow{}$(\ref{it:HpqCH1})} It follows from the equivalence (\ref{it:HpqH3})${}\Leftrightarrow{}$(\ref{it:HpqH7}) of Theorem~\ref{thm:HpqHelly} as follows. Suppose that $G$ is not hereditary $(p,q)$-clique-Helly and let $G'$ be an induced subgraph of $G$ which is not $(p,q)$-clique-Helly. In particular, $\mathcal C(G')$ is not strong $(p,q)$-Helly and, by Theorem~\ref{thm:HpqHelly}, has some partial subhypergraph $\mathcal H$ isomorphic to $\mathcal J_{p+1,q,s}$ for some $s\in\{0,\ldots,q-1\}$.  Since $\mathcal C(G')$ is a partial subhypergraph of $\mathcal C(G)$, $\mathcal H$ is also a partial subhypergraph of $\mathcal C(G)$. This means that $\mathcal C(G)$ is not strong $(p,q)$-Helly; i.e., $G$ is not strong $(p,q)$-clique-Helly.

\mbox{(\ref{it:HpqCH1})${}\Rightarrow{}$(\ref{it:HpqCH8})} It follows from Lemma~\ref{lem:pqs-ocular}.

\mbox{(\ref{it:HpqCH8})${}\Rightarrow{}$(\ref{it:HpqCH7})} Suppose \eqref{it:HpqCH7} does not hold; i.e., there is some $s\in\{0,\ldots,q-1\}$, some $((p+1)(q-s)+s)$-subset $U$ of $V(G)$, and some $p+1$ pairwise disjoint $(q-s)$-subsets $T_1,\ldots,T_{p+1}$ of $U$ such that $U-T_1,\ldots,U-T_{p+1}$ are cliques of $G$ and for each $i\in\{1,\ldots,p+1\}$ and each $v\in T_i$ there is some vertex $\trans{v}$ which is complete to $U-T_i$ but nonadjacent to $v$.

Consider first the case $p=1$. If, for each $i\in\{1,2\}$, no vertex $v\in T_i$ is complete to $U-T_i$, then $U$ induces a $(p+1,q,s)$-ocular in $G$. If, on the contrary, there is some $i\in\{1,2\}$ and some $v\in T_i$ that is complete to $U-T_i$, then $(U-T_i)\cup\{v,\trans{v}\}$ induces a $(p+1,q,q-1)$-ocular in $G$. In either case, (\ref{it:HpqCH8}) does not hold.

From now on, we assume that $p\geq 2$. Observe that this means that $U$ is a clique of $G$ because every pair of vertices of $U$ is contained in at least one of the cliques $U-T_1,\ldots,U-T_{p+1}$ of $G$. For each $i\in\{1,\ldots,p+1\}$ and each $v\in T_i$, let $W(v)=\{w\in V(G):\,w\text{ is complete to }T-U_i\text{ and nonadjacent to }v\}$. Suppose first that there are two vertices $x,y\in T_j$ such that $W(x)$ and $W(v)$ are inclusion-wise incomparable. By symmetry, let $j=1$ and let $x'$ and $y'$ be two vertices complete to $U-T_1$ such that $x'$ is adjacent to $y$ but not to $x$, and $y'$ is adjacent to $x$ but not to $y$. Thus, for any choice of a vertex $v_i\in T_i$ for each $i\in\{2,\ldots,p+1\}$, the set $\{x,y,v_2,\ldots,v_p\}\cup(T_{p+1}-\{v_{p+1}\})\cup Z\cup\{x',y',\trans{v_2},\ldots,\trans{v_p}\}$ induces a $(p+1,q,q-1)$-ocular in $G$ and (\ref{it:HpqCH8}) does not hold. Hence, for each $i\in\{1,\ldots,p+1\}$, we assume, without loss of generality, that the family $\{W(v):\,v\in T_i\}$ is a chain under inclusion, let $v_i\in T_i$ such that $W(v_i)$ is the minimum element of the chain and let $w_i=\trans{v_i}$. By construction, for each $i\in\{1,\ldots,p+1\}$, $w_i$ is complete to $U-T_i$ and anticomplete to $T_i$. Thus, the set $U\cup\{\trans{v_1},\ldots,\trans{v_{p+1}}\}$ induces a $(p+1,q,s)$-ocular in $G$ and (\ref{it:HpqCH8}) does not hold either. This completes the proof of (\ref{it:HpqCH8})${}\Rightarrow{}$(\ref{it:HpqCH7}) and thus of the theorem.\end{proof}

By combining the equivalences (\ref{it:HpqCH1})${}\Leftrightarrow{}$(\ref{it:HpqCH3}) of Theorem~\ref{thm:HpqCHelly} and (\ref{it:HpqH1})${}\Leftrightarrow{}$(\ref{it:HpqH3}) of Theorem~\ref{thm:HpqHelly}, we obtain the following consequence which is crucial for our derivation of a polynomial-time recognition algorithm.

\begin{cor}\label{cor:HpqCHelly} For each pair of positive integers $p$ and $q$, a graph $G$ is hereditary $(p,q)$-clique-Helly if and only if $\mathcal C(G)$ is hereditary $(p,q)$-Helly.\end{cor}

As a consequence of the above result, the fact that $\mathcal C(G)$ is a simple hypergraph, and Lemma~\ref{lem:H1qHelly}, we have the following characterizations of hereditary $(1,q)$-clique-Helly graphs.

\begin{cor}\label{cor:H1qCHelly} If $q$ is a positive integer, then the following statements are equivalent for each graph $G$:
\begin{enumerate}[(i)]
 \item\label{it:H1qCH1} $G$ is hereditary $(1,q)$-clique-Helly
 \item\label{it:H1qCH2} $G$ has at most one maximal clique of cardinality at least $q$.
 \item\label{it:H1qCH3} the union of all the $q$-cliques of $G$ is empty or a maximal clique of $G$.
 \item\label{it:H1qCH4} the union of all the $q$-cliques of $G$ is a (possibly empty) clique of $G$.
\end{enumerate}\end{cor}

We now give time bounds for the recognition of hereditary $(p,q)$-clique-Helly graphs. Our bounds represent an improvement upon the bound (\ref{it:HpCH-rec1}) of Theorem~\ref{thm:HpCHelly-rec} proved in \cite{MR2404220} for hereditary $p$-clique-Helly graphs. Moreover, for the recognition of hereditary clique-Helly graphs (i.e., hereditary $(2,1)$-clique-Helly graphs), our bound matches the $O(m^2+n)$-time bound given in~\cite{MR2321859}, which is the best currently known. In fact, our algorithm for the case $q=1$ is a generalization of the approach for the case $(p,q)=(2,1)$ used in~\cite{MR2321859}.

\begin{thm} If $q$ is any fixed positive integer, then the recognition problem for hereditary $(p,q)$-clique-Helly graphs, where $p$ is part of the input, can be solved in:
\begin{enumerate}[(i)]
 \item\label{it:HpqCH-rec1} $O(m+n)$ time if $p=1$ and $q=1$;
 \item\label{it:HpqCH-rec2} $O(qm^{q/2}+n)$ time if $p=1$ and $q\geq 2$;
 \item\label{it:HpqCH-rec3} $O(m^{p/2+1}+p\omega m^{(p+1)/2}+n)$ time if $p\geq 2$ and $q=1$;
 \item\label{it:HpqCH-rec4} $O\bigl(qm^{q/2}+\omega m^{q/2}\binom{N_q}p+n\bigr)$ time if $p\geq 2$ and $q\geq 2$, where $N_q$ is the number of $q$-cliques of the input graph.
\end{enumerate}
In particular, if $p$ and $q$ are both fixed, bound (\ref{it:HpqCH-rec4}) above becomes $O(\omega m^{(p+1)q/2}+n)$.
\end{thm}
\begin{proof} Let $G$ be the input graph and let $\mathcal C$ be the clique hypergraph of $G$.

Suppose that $p=1$. By virtue of Corollary~\ref{cor:H1qCHelly}, $G$ is hereditary $(p,q)$-clique-Helly if and only if the union of all the $q$-cliques of $G$ is a clique of $G$. In particular, if $q=1$, then $G$ is hereditary $(p,q)$-clique-Helly if and only if $G$ is a complete graph, which can be decided in $O(m+n)$ time. If, on the contrary, $q\geq 2$, then we can enumerate all the $q$-cliques of $G$ while computing its union $U$ in $O(qm^{q/2}+n)$ time (by Theorem~\ref{thm:qcliques}) and determine whether $U$ is a clique of $G$ in additional $O(m+n)$ time; hence, it can be decided whether $G$ is hereditary $(p,q)$-clique-Helly in $O(qm^{q/2}+n)$ time.

From now on, we suppose that $p\geq 2$. Since $G$ is hereditary $(p,q)$-clique-Helly if and only if each component of $G$ is and we can compute the components of $G$ in $O(m+n)$ time, we assume, without loss of generality, that $G$ is connected.

We consider first the case $q=1$. By the equivalences (\ref{it:HpCH1})${}\Leftrightarrow{}$(\ref{it:HpCH2}) of Theorem~\ref{thm:HpCHelly} and (\ref{it:HpH3})${}\Leftrightarrow{}$(\ref{it:HpH7}) of Theorem~\ref{thm:HpHelly} and Remark~\ref{rmk:nontrivial}, $G$ is hereditary $p$-clique-Helly if and only if $\core(\mathcal C_P^\cup)\cap P\neq\emptyset$ for every $(p+1)$-clique $P$ of $G$. We can verify the latter by proceeding as in the proof of Theorem~\ref{thm:recog-pq-CH} for the case where $p\geq 2$ and $q=1$. The only difference is that whenever we reach a leaf node of $T_2$ corresponding to a $(p+1)$-clique $P$, we not only compute $\core(\mathcal C^\cup_P)$ but also compute its intersection with $P$. These additional computations account for additional $O(p m^{(p+1)/2})$ time only and hence the bound $O(m^{p/2+1}+p\omega m^{(p+1)/2}+n)$ remains unaltered.

It only remains to consider the case where $p\geq 2$ and $q\geq 2$. By Corollary~\ref{cor:HpqCHelly}, $G$ is hereditary $(p,q)$-clique-Helly if and only if $\mathcal C$ is hereditary $(p,q)$-Helly, and the latter can be decided in $O\bigl(\omega m^{q/2}\binom{N_q}p\bigr)$ time where $N_q=O(m^{q/2})$ after enumerating all $q$-cliques of $G$ in $O(qm^{q/2})$ time, proceeding as in the last paragraph of the proof of Theorem~\ref{thm:recog-pq-CH} by making Lemma~\ref{lem:technical2} play the role of Lemma~\ref{lem:technical}.\end{proof}

In~\cite{MR2365055}, it was proved that $K_{p+q}$-free graphs are $(p,q)$-clique Helly. The following is an immediate consequence.

\begin{thm}[\cite{MR2365055}]\label{thm:Kpq-free} If $p$ and $q$ are positive integers, every $K_{p+q}$-free graph is hereditary $(p,q)$-clique-Helly.\end{thm}

Notice that the above result also follows from our Theorem~\ref{thm:HpqCHelly} because each $(p,q,s)$-ocular has some $(p+q)$-clique. We conclude this section by extending the hardness result contained in statement (\ref{it:HpCH-rec2}) of Theorem~\ref{thm:HpCHelly-rec} as follows.

\begin{cor} The recognition problem for hereditary $(p,q)$-clique-Helly graphs, for positive integers $p$ and $q$, is NP-hard if $p$ or $q$ is part of the input.\end{cor}
\begin{proof} We first prove the NP-hardness if $p$ is part of the input by a reduction from the problem of recognizing hereditary $p$-clique-Helly graphs (which in turn is NP-hard by statement (\ref{it:HpCH-rec2}) of Theorem~\ref{thm:HpCHelly-rec}). Let $G$ be a graph and let $p\geq 1$. Let $G'$ be the graph that arises from $G$ by adding $q-1$ universal vertices. Clearly, $C(G')$ arises from $C(G)$ by adding $q-1$ columns filled with $1$'s. Hence, by virtue of the equivalence (\ref{it:HpqCH1})${}\Leftrightarrow{}$(\ref{it:HpqCH6}) of Theorem~\ref{thm:HpqCHelly}, $G$ is hereditary $p$-clique-Helly if and only if $G'$ is hereditary $(p,q)$-clique-Helly. This proves the correctness of the reduction and hence the NP-hardness.

Assume $p\geq 1$ fixed and let $H$ be a $(p+1)$-ocular. Observe that the clique number of $H$ is $p+1$.  We now prove the NP-hardness when $q$ is part of the input by reduction from the \textsc{Clique} problem~\cite{MR0378476}. Let a graph $G$ and a positive integer $k$ be given. Without loss of generality assume that $V(G)\cap V(H)=\emptyset$ and let $G'$ be the \emph{join} of $G$ and $H$ (i.e., the graph that arises from the disjoint union of $G$ and $H$ by adding all the possible edges joining a vertex of $G$ to a vertex of $H$). We let $q=k+1$ and claim that $G$ has a $k$-clique if and only if $G'$ is hereditary $(p,q)$-clique-Helly. On the one hand, if $G$ has a $(q-1)$-clique, then $G'$ contains the join of $H$ and $K_{q-1}$, which is a $(p,q,q-1)$-ocular, as an induced subgraph and, by Lemma~\ref{lem:pqs-ocular}, $G'$ is not hereditary $(p,q)$-clique-Helly. On the other hand, if $G$ has no $(q-1)$-clique, then $G'$ is $K_{p+q}$-free and, by Theorem~\ref{thm:Kpq-free}, $G'$ is hereditary $(p,q)$-clique-Helly. This proves the claim and hence the correctness of the reduction and the NP-hardness.\end{proof}

\section{The $(p,q)$-biclique-Helly property}\label{sec:HpqBH}

In this section, we give polynomial-time recognition algorithms and prove different characterizations for $(p,q)$-biclique-Helly graphs and hereditary $(p,q)$-biclique-Helly graphs which are analogous to those for $(p,q)$-clique-Helly and hereditary $(p,q)$-clique-Helly graphs proved in the preceding section. In addition, we prove that the recognition problems of $(p,q)$-biclique-Helly graphs and of hereditary $(p,q)$-biclique-Helly graphs are co-NP-complete when $p$ or $q$ are part of the input.

We begin with some preliminary properties of complete bipartite graphs and bicliques. Recall that we regard edgeless graphs as complete bipartite graphs. The following is a straightforward observation.

\begin{rmk}\label{rmk:coP3-C3} A graph $G$ is complete bipartite if and only if $G$ is $\{\overline{P_3},K_3\}$-free.\end{rmk}

The next result shows that $\overline{P_3}$ and $K_3$ cannot occur as induced subgraphs an odd number of times in total among four vertices.

\begin{lem}\label{lem:abcd} If $a,b,c,d$ are (not necessarily pairwise different) vertices of a graph $G$, then precisely an even number of the sets $\{a,b,c\}$, $\{a,b,d\}$, $\{a,c,d\}$, and $\{b,c,d\}$ are bicliques of $G$.\end{lem}
\begin{proof} If there are repeated vertices among $a,b,c,d$, the lemma holds trivially. Thus, we assume, without loss of generality, that $a,b,c,d$ are pairwise different. If $x_1,\ldots,x_k$ are vertices of $G$, we denote by $e(x_1,\ldots,x_k)$ the number of edges of $G[\{x_1,\ldots,x_k\}]$. Since $e(a,b,c)+e(a,b,d)+e(a,c,d)+e(b,c,d)=2e(a,b,c,d)$, there are an even number of even numbers among $e(a,b,c)$, $e(a,b,d)$, $e(a,c,d)$, and $e(b,c,d)$. As three different vertices $x,y,z$ induce $\overline{P_3}$ or $K_3$ precisely when $e(x,y,z)$ is odd, the result follows by Remark~\ref{rmk:coP3-C3}.\end{proof}

The following result about bicliques will be useful in what follows.

\begin{lem}\label{lem:B1B2} Let $B$ and $B'$ be two bicliques of a graph $G$ such that $B\cap B'\neq\emptyset$. If $v\in B$ is such that $B'\cup\{v\}$ is not a biclique, then there is some vertex $w\in B'-B$ such that $(B\cap B')\cup\{v,w\}$ is not a biclique. Moreover, $\{v,x,w\}$ is not a biclique for any $x\in B\cap B'$.\end{lem}
\begin{proof} Let $v\in B$ such that $B'\cup\{v\}$ is not a biclique and suppose, for a contradiction, that $\{v,y,w\}$ is a biclique for each $y\in B\cap B'$ and each $w\in B'-B$. Thus, $\{v,w,w'\}$ is also a biclique for every $w,w'\in B'-B$ by virtue of Lemma~\ref{lem:abcd} because for any $y\in B\cap B'$, the sets $\{v,y,w\}$, $\{v,y,w'\}$, and $\{y,w,w'\}$ are bicliques. But, since $\{v,y,y'\}$ is a biclique for each $y,y'\in B\cap B'$ and $B'$ is a biclique, every $3$-subset of $B'\cup\{v\}$ is a biclique and Remark~\ref{rmk:coP3-C3} implies that $B'\cup\{v\}$ is a biclique, a contradiction. This contradiction arose from supposing that $\{v,y,w\}$ was a biclique for each $y\in B\cap B'$ and each $w\in B'-B$. Hence, there is some $y\in B\cap B'$ and some $w\in B'-B$ such that $\{v,y,w\}$ is not a biclique. Let $x\in B\cap B'$. Since $\{v,y,x\}$ and $\{w,y,x\}$ are bicliques (because they are contained in $B$ and $B'$, respectively) and $\{v,y,w\}$ is not a biclique, Lemma~\ref{lem:abcd} implies that $\{v,x,w\}$ is not a biclique. This proves that $\{v,x,w\}$ is not a biclique for any $x\in B\cap B'$.\end{proof}

\subsection{$(p,q)$-biclique-Helly graphs}

In this subsection, we address the problems of characterizing and recognizing $(p,q)$-biclique-Helly graphs. For that purpose, we introduce some definitions. We denote by $\mathcal B(G)$ the \emph{biclique hypergraph} of a graph $G$, which is the hypergraph whose edge family is the set of maximal bicliques of $G$. By definition, a graph $G$ is $(p,q)$-biclique-Helly if and only if $\mathcal B(G)$ is $(p,q)$-Helly. A \emph{$2$-labeled graph} is a graph whose vertices are labeled using at most two different labels in such a way that each vertex is given exactly one label. Let $H$ be a $2$-labeled graph. We say that two different vertices of $H$ are \emph{bicompatible} if either: (1) they are adjacent and have different labels; or (2) they are nonadjacent and have the same label. A \emph{biuniversal vertex} of $H$ is a vertex of $H$ that is bicompatible with every other vertex of $H$. Equivalently, $v$ is a biuniversal if it is adjacent precisely to those vertices whose label is different from that of $v$.

\begin{rmk} Let $H$ be a $2$-labeled graph. Any set of pairwise bicompatible vertices of $H$ is a biclique. If $a,b,c$ are three different vertices of $H$ such that $a$ is bicompatible with $b$ and $c$, then the following holds: $\{a,b,c\}$ is a biclique if and only if $b$ and $c$ are bicompatible.\end{rmk}

Let $G$ be a graph and let $P$ be a nonempty biclique of $G$. Let $\{X,Y\}$ be the only bipartition of $G[P]$ as a complete bipartite graph. The \emph{bicompletion} of $P$ in $G$, denoted $\bicomp_G(P)$ is the $2$-labeled graph that arises by labeling with $X$ and $Y$ the vertices of the subgraph of $G$ induced by those vertices $v$ such that $P\cup\{v\}$ is a biclique of $G$ as follows: if $v$ belongs to $X$ or $Y$, then $v$ is labeled with $X$ or $Y$, respectively, whereas if $v\notin P$ then $v$ is labeled so as to make $v$ biuniversal in $G[P\cup\{v\}]$ (i.e., bicompatible with every vertex of $P$).

The result below, which resembles Lemma~\ref{lem:CP}, will be useful for both characterizing and recognizing $(p,q)$-biclique-Helly graphs.

\begin{lem}\label{lem:BP} Let $G$ be a graph and let $P$ be a nonempty biclique of $G$. If $\mathcal B$ is the biclique hypergraph of $G$ and $v$ is a vertex of $G$, then the following statements are equivalent:
\begin{enumerate}[(i)]
 \item\label{it:BP1} $v\in\core(\mathcal B_P)$;
 \item\label{it:BP2} $P\cup\{v,w\}$ is a biclique of $G$ for each vertex $w$ of $G$ such that $P\cup\{w\}$ is a biclique of $G$;
 \item\label{it:BP3} $v$ is a biuniversal vertex of the bicompletion of $P$ in $G$.
\end{enumerate}
In particular, $\core(\mathcal B_P)$ is a biclique of $G$ containing $P$ and can be computed in $O(m+n)$ time.\end{lem}
\begin{proof}\mbox{(\ref{it:BP1})${}\Rightarrow{}$(\ref{it:BP2})} Suppose \eqref{it:BP2} does not hold; i.e., there is some vertex $w$ such that $P\cup\{v,w\}$ is not a biclique of $G$ but $P\cup\{w\}$ is a biclique. Thus, if we let $B$ be any maximal biclique of $G$ containing $P\cup\{w\}$, then $v\notin B$ and (\ref{it:BP1}) does not hold.

\mbox{(\ref{it:BP2})${}\Rightarrow{}$(\ref{it:BP3})} Suppose \eqref{it:BP3} does not hold. If $P\cup\{v\}$ is not a biclique, then (\ref{it:BP2}) does not hold. Thus, we assume, without loss of generality, that $P\cup\{v\}$ is a biclique. Since $v$ is not biuniversal in the bicompletion of $P$ in $G$, there is some vertex $w$ such that $P\cup\{w\}$ is also a biclique but $v$ and $w$ are not bicompatible. By construction, none of $v$ and $w$ belongs to $P$ and given any $p\in P$, $p$ is bicompatible with $v$ and $w$. Thus, $\{v,w,p\}$ is not a biclique of $G$ and (\ref{it:BP2}) does not hold.

\mbox{(\ref{it:BP3})${}\Rightarrow{}$(\ref{it:BP1})} Suppose \eqref{it:BP1} does not hold; i.e., $v\notin\core(\mathcal B_P)$. If $P\cup\{v\}$ is not a biclique, then $v$ is not even a vertex of the bicompletion of $P$ in $G$ and (\ref{it:BP3}) does nod hold. Thus, we assume, without loss of generality, that $P\cup\{v\}$ is a biclique. Since $v\notin\core(\mathcal B_P)$, there is some maximal biclique $B$ of $G$ containing $P$ such that $v\notin B$. Hence, Lemma~\ref{lem:B1B2} implies that there is some vertex $w\in B-P$ such that $\{v,w,p\}$ is not a biclique of $G$ for any $p\in P$. Since, by construction, $p$ is bicompatible with $v$ and $w$, necessarily $v$ is not bicompatible with $w$ and (\ref{it:BP3}) does not hold.

\smallskip

We conclude that indeed (\ref{it:BP1}), (\ref{it:BP2}), and (\ref{it:BP3}) are equivalent. Notice that because of (\ref{it:BP1})${}\Leftrightarrow{}$(\ref{it:BP3}), the vertices of $\core(\mathcal B_P)$ are pairwise bicompatible in the completion of $P$ in $G$ and, consequently, $\core(\mathcal B_P)$ is a biclique of $G$. That $\core(\mathcal B_P)$ can be computed in $O(m+n)$ time follows from (\ref{it:BP1})${}\Leftrightarrow{}$(\ref{it:BP3}). Clearly, we can find the only bipartition $\{X,Y\}$ of $G[P]$ as a complete bipartite graph in $O(m+n)$ time. Moreover, in additional $O(m+n)$ time, it is possible to count the number of neighbors in $X$ and in $Y$ for each vertex of $G$ by traversing the neighborhoods of the vertices in $X$ and in $Y$ once. Then, in additional $O(m+n)$ time, it is possible to determine $\bicomp_G(P)$ and also its biuniversal vertices (by counting the number of neighbors labeled with $X$ and with $Y$ in a similar way to what we have just done).\end{proof}

The above lemma and Corollary~\ref{cor:pqHelly} for $\mathcal H=\mathcal B(G)$ lead immediately to the following characterizations of $(p,q)$-biclique-Helly graphs.

\begin{cor}\label{cor:pqBHelly} For each pair of positive integers $p$ and $q$, a graph $G$ is $(p,q)$-biclique-Helly if and only if, for every family $\mathcal D$ of $p+1$ pairwise different $q$-subsets of $V(G)$ such that each of the unions $P_1,\ldots,P_{p+1}$ of $p$ pairwise different members of $\mathcal D$ is a biclique of $G$, there are at least $q$ vertices of $G$ that are biuniversal in each of the bicompletions of $P_1,\ldots,P_{p+1}$ in $G$.\end{cor}
\begin{proof} By Corollary~\ref{cor:pqHelly}, $G$ is $(p,q)$-biclique-Helly if and only if $\mathcal B^\cup_\mathcal D$ has $q^+$-core for each nontrivial $(p+1,q)$-basis $\mathcal D$ of $\mathcal B$. Since $\core(\mathcal B^\cup_\mathcal D)=\core(\mathcal B_{P_1})\cap\cdots\cap\core(\mathcal B_{P_{p+1}})$ where $P_1,\ldots,P_{p+1}$ are the support sets of $\mathcal D$ and, by Lemma~\ref{lem:BP}, $\core(\mathcal B_{P_i})$ is the set of biuniversal vertices of the bicompletion of $P_i$ in $G$ for each $i\in\{1,\ldots,p+1\}$, the corollary follows.\end{proof}

We now consider the case $p\geq 2$ and we will derive a characterization of $(p,q)$-biclique-Helly graphs in terms of what we call \emph{biexpansions}, which resembles Theorem~\ref{thm:pqcliqueHelly}. Let $\mathcal D$ be a family of $p+1$ pairwise different $q$-bicliques of $G$ whose union $U$ is a biclique of $G$ and let $\{X,Y\}$ be the only bipartition of $G[U]$ as a complete bipartite graph. We define the $(p+1,q)$-\emph{biexpansion} of $\mathcal D$ in $G$ as the subgraph of $G$ induced by those vertices $v$ such that $P\cup\{v\}$ is a biclique for some union $P$ of $p$ pairwise different members of $\mathcal D$ and where each such vertex $v$ is labeled with $X$ or $Y$ as follows: if $v$ belongs to $X$ or $Y$, then $v$ is labeled with $X$ or $Y$, respectively, whereas if $v\notin U$, then $v$ is labeled so as to make $v$ biuniversal in $G[P\cup\{v\}]$ (i.e., bicompatible with every vertex of $P$.) It is easy to see that no vertex $v$ is labeled with $X$ for some $P$ and with $Y$ for a different $P$. In fact, if there were any such vertex $v$, then $v\notin U$ and $v$ would be simultaneously adjacent and nonadjacent to each vertex of $P_1\cap P_2$ where $P_1$ and $P_2$ are two unions of $p$ pairwise different members of $\mathcal D$ each and, consequently, $\vert P_1\cap P_2\vert\geq q$.

The following is the characterization of $(p,q)$-biclique-Helly graphs, where $p\geq 2$, in terms of biexpansions.

\begin{thm}\label{thm:biexp} For each integer $p\geq 2$ and each positive integer $q$, a graph $G$ is $(p,q)$-biclique-Helly if and only if each of the following statements holds:
\begin{enumerate}[(i)]
\item every $(p+1,q)$-biexpansion in $G$ has at least $q$ biuniversal vertices;
\item\label{it:p=2} if $p=2$ then there are no three $q$-bicliques $B_1,B_2,B_3$ of $G$ such that $B_1\cup B_2$, $B_2\cup B_3$, and $B_3\cup B_1$ are bicliques of $G$ but $B_1\cup B_2\cup B_3$ is not a biclique of $G$.
\end{enumerate}\end{thm}
\begin{proof} Let $\mathcal B$ be the biclique hypergraph of $G$. We consider first the case where there is some nontrivial $(p+1,q)$-basis $\mathcal D$ of $\mathcal B$ whose union is not a biclique. Necessarily, (\ref{it:p=2}) does not hold; i.e., $p=2$ and there are three $q$-bicliques $B_1,B_2,B_3$ of $G$ such that $B_1\cup B_2$, $B_2\cup B_3$, $B_3\cup B_1$ are bicliques of $G$ but $B_1\cup B_2\cup B_3$ is not a biclique of $G$. Hence, there is a vertex $b_i\in B_i$ for each $i\in\{1,2,3\}$ such that $\{b_1,b_2,b_3\}$ is not a biclique of $G$. If $G$ were $(p,q)$-Helly, then, by Corollary~\ref{cor:pqBHelly}, there should be some vertex $u$ belonging to the bicompletion of each of $B_1\cup B_2$, $B_2\cup B_3$, and $B_3\cup B_1$ in $G$. In particular, $\{b_1,b_2,u\}$, $\{b_2,b_3,u\}$, $\{b_3,b_1,u\}$ would be bicliques of $G$ and, by Lemma~\ref{lem:abcd}, $\{b_1,b_2,b_3\}$ would also be a biclique of $G$. This contradiction proves that $G$ is not $(p,q)$-biclique-Helly. Hence, if $\mathcal B$ has some nontrivial $(p+1,q)$-basis whose union is not a biclique of $G$, then the theorem holds for $G$. Therefore, from now on, we assume that, for each nontrivial $(p+1,q)$-basis of $\mathcal B$, its union is a biclique of $G$ and thus statement (\ref{it:p=2}) does not hold.

By Corollary~\ref{cor:pqHelly}, $G$ is $(p,q)$-biclique-Helly if and only if $\mathcal B^\cup_\mathcal D$ has a $q^+$-core for each nontrivial $(p+1,q)$-basis $\mathcal D$ of $\mathcal B$. Since we are assuming that, for each nontrivial $(p+1,q)$-basis of $\mathcal B$, its union is a biclique of $G$, in order to prove the theorem it is enough to show that, for each family $\mathcal D$ of $p+1$ pairwise different $q$-bicliques of $G$ whose union is a biclique of $G$, the core of $\mathcal B^\cup_\mathcal D$ coincides with the set of the biuniversal vertices in the biexpansion of $\mathcal D$ in $G$.

Let $v\in\core(\mathcal B^\cup_\mathcal D)$ and let $v'$ be any vertex in the biexpansion of $\mathcal D$ in $G$. By definition, $v'$ is a vertex of $\bicomp_G(P)$ for some support set $P$ of $\mathcal D$. Since $v\in\core(\mathcal B^\cup_\mathcal D)$, $v\in\core(\mathcal B_P)$ which, by Lemma~\ref{lem:BP}, implies that $v$ is bicompatible with $v'$ in $\bicomp_G(P)$. By definition of biexpansions, $v$ and $v'$ are also bicompatible in the biexpansion of $\mathcal D$ in $G$. This proves that every vertex of $\core(\mathcal B^\cup_\mathcal D)$ is biuniversal in the biexpansion of $\mathcal D$ in $G$. Conversely, let $v$ be a biuniversal vertex in the biexpansion of $\mathcal D$ in $G$. Since, by construction, the vertices of $U$ are pairwise bicompatible and, by assumption, $v$ is bicompatible with every vertex of $U$, then $U\cup\{v\}$ is a biclique of $G$. In particular, $v$ is a vertex of $\bicomp_G(P)$ for each support set $P$ of $\mathcal D$. Moreover, since $v$ is bicompatible with every other vertex in the biexpansion of $\mathcal D$ in $G$, the definition of expansions implies that $v$ is also biuniversal in $\bicomp_G(P)$. Hence, by Lemma~\ref{lem:BP}, $v\in\core(\mathcal B_P)$ for each support set $P$ of $\mathcal D$. By definition, this means that $v\in\core(\mathcal B^\cup_\mathcal D)$. This proves that every biuniversal vertex in the biexpansion of $\mathcal D$ belongs to $\core(\mathcal B^\cup_\mathcal D)$ and completes the proof of the theorem.\end{proof}

We now turn to the problem of recognizing $(p,q)$-biclique-Helly graphs. First, we prove the following lemma.

\begin{lem}\label{lem:11,12,21} If $G$ is a graph and $(p,q)=(1,1)$, $(1,2)$, or $(2,1)$, then $G$ is $(p,q)$-biclique-Helly if and only if $G$ is a complete bipartite graph.\end{lem}
\begin{proof} The `if' direction is clear. Conversely, suppose that $G$ is $(p,q)$-biclique-Helly. If $(p,q)=(2,1)$, then statement (\ref{it:p=2}) of Theorem~\ref{thm:biexp} and Remark~\ref{rmk:coP3-C3} imply that $G$ is a complete bipartite graph. Thus, assume, without loss of generality, that $(p,q)=(1,1)$ or $(1,2)$. We assume also that $G$ has at least one edge (since otherwise $G$ is a complete bipartite graph). By Remark~\ref{rmk:1qHelly}, there is at least one vertex $u$ of $G$ that belongs to $\core(\mathcal B_P)$ for each $q$-biclique $P$ of $G$. In particular, $\{u,v,w\}$ is a biclique of $G$ for every two other vertices $v,w\in V(G)$, which immediately implies that $G$ is a complete bipartite graph with bipartition $\{N_G(v),V(G)-N_G(v)\}$.\end{proof}

We show that $(p,q)$-biclique-Helly graphs can be recognized in polynomial time for fixed $p$ and $q$. Recall that $\psi$ denotes the cardinality of the largest biclique of the input graph.

\begin{thm}\label{thm:recog-pq-BH} If $q$ is any fixed positive integer, then the recognition problem for $(p,q)$-biclique-Helly graphs, where $p$ is part of the input, can be solved in:
\begin{enumerate}[(i)]
\item\label{it:pqBH1-rec} $O(m+n)$ time if $(p,q)=(1,1)$, $(1,2)$, or $(2,1)$;
\item\label{it:pqBH2-rec} $O(qn^q+(m+n)N_q)$ time if $p=1$ and $q\geq 3$;
\item\label{it:pqBH3-rec} $O\bigl(qn^q+(\psi N_q+m+n)\binom{N_q}p\bigr)$ time if $p\geq 2$ where $N_q$ is the number of $q$-bicliques of the input graph.
\end{enumerate}
In particular, if $p$ and $q$ are both fixed, then the bounds (\ref{it:pqBH2-rec}) and (\ref{it:pqBH3-rec}) above become $O((m+n)n^q)$ and $O(\psi n^{(p+1)q})$, respectively.\end{thm}
\begin{proof} Let $G$ be the input graph and let $\mathcal B$ denote the biclique hypergraph of $G$. The validity of bound (\ref{it:pqBH1-rec}) follows directly from Lemma~\ref{lem:11,12,21}. If $q>n$, then $G$ has no $q$-bicliques and thus $G$ is $(p,q)$-biclique-Helly. Hence, we assume, without loss of generality that $q\leq n$.

Suppose that $p=1$. By Remark~\ref{rmk:1qHelly}, $G$ is $(p,q)$-biclique-Helly if and only if the family of cores of $\mathcal B_P$, where $P$ varies over the $q$-bicliques of $G$, is empty or has a $q^+$-core. Thus, the bound (\ref{it:pqBH2-rec}) follows because we can enumerate all the $q$-bicliques $P$ of $G$ in $O(qn^q)$ time, compute $\core(\mathcal B_P)$ for each such $P$ in $O(m+n)$ time (by Lemma~\ref{lem:BP}), and compute the intersection of all such cores in $O\bigl(qN_q+n)=O(nN_q)$ time.

Suppose now that $p\geq 2$. Since we can enumerate all the $N_q$ $q$-bicliques of $G$, in at most $O(qn^q)$ time and Lemma~\ref{lem:BP} ensures that we can compute $\core(\mathcal B_P)$ in $O(m+n)$ time for any given subset $P$ of $V(G)$, it follows from Lemma~\ref{lem:technical} that it can be decided whether $G$ is $(p,q)$-biclique-Helly in $O\bigl(qn^q+(\psi N_q+m+n)\binom{N_q}p\bigr)$ time.\end{proof}

We conclude this subsection proving that the recognition problem for $(p,q)$-biclique-Helly graphs is co-NP-complete if $p$ or $q$ is part of the input. First, we prove that the problem is in co-NP.

\begin{lem}\label{lem:pqBH-coNP} The recognition problem of $(p,q)$-biclique-Helly graphs, for positive integers $p$ and $q$, is in co-NP even if $p$ or $q$ are part of the input.\end{lem}
\begin{proof} Let $G$ be a graph which is not $(p,q)$-biclique-Helly. Since $\mathcal B(G)$ is not $(p,q)$-Helly and, consequently, not hereditary $(p,q)$-Helly, Theorem~\ref{thm:HpqHelly} implies that $\mathcal B(G)$ contains $\mathcal J_{p+1,q,s}$ as a partial subhypergraph for some $s\in\{0,\ldots,q-1\}$ and, in particular, $G$ has at least $p+q$ vertices. Moreover, by Corollary~\ref{cor:pqBHelly}, there is a certificate of the fact that $G$ is not $(p,q)$-biclique-Helly in the form of a family $\mathcal D$ of $p+1$ pairwise different $q$-subsets of $V(G)$ such that each of the unions $P_1,\ldots,P_{p+1}$ of $p$ pairwise different members of $\mathcal D$ is a biclique of $G$ and such that there are no $q$ vertices of $G$ that are biuniversal in each of the bicompletions of $P_1,\ldots,P_{p+1}$ in $G$. Since, given a collection $\mathcal D$ of $p+1$ pairwise different $q$-subsets of $V(G)$, the validity of $\mathcal D$ as a certificate can be decided in polynomial time of the size of $G$, the recognition problem of $(p,q)$-biclique-Helly graphs is in co-NP.\end{proof}

\begin{thm}\label{thm:pq-BH-coNP-pvar} The recognition problem of $(p,q)$-biclique-Helly graphs, for positive integers $p$ and $q$, is co-NP-complete if $p$ is part of the input.\end{thm}
\begin{proof} By Lemma~\ref{lem:pqBH-coNP}, the problem belongs to co-NP. We use a reduction from 3,4-SAT~\cite{MR739601}, a variation of the \textsc{Satisfiability} problem~\cite{MR0378476} in which every clause has exactly three literals and each variable occurs in at most four clauses. Let $\mathcal F=\{C_1,\ldots,C_m\}$ be an instance of 3,4-SAT where $C_i = \{\ell_{i,1}, \ell_{i,2}, \ell_{i,3}\}$ for each $i\in\{1,\ldots,m\}$. Without loss of generality, we assume that $m\geq 7$ and no clause contains a literal and its negation.

Let $G$ be the graph with $4m+q-1$ vertices defined as follows. The vertex set of $G$ is $Y\cup U\cup Q$ where $Y=\{y_1,\ldots,y_m\}$ is a clique, $U=U_1\cup\cdots\cup U_m$ is such that $U_i=\{u_{i,1},u_{i,2},u_{i,3}\}$ is a clique for each $i\in\{1,\ldots,m\}$, and $Q$ is a stable set consisting of $q-1$ vertices. Moreover, for each two different $i,j\in\{1,\ldots,m\}$, the following holds: (i) $y_i$ is complete to $U_j$ and anticomplete to $U_i$; and (ii) for every $a,b\in\{1,2,3\}$, $u_{i,a}$ is adjacent to $u_{j,b}$ if and only if the literal $\ell_{i,a}$ is the negation of $\ell_{j,b}$. For each $i\in\{1,\ldots,m\}$ and each $a\in\{1,2,3\}$, we call $\ell_{i,a}$ \emph{the literal associated with} $u_{i,a}$. Finally, $Q$ is complete to $U$ and anticomplete to $Y$. Let $p=m-1$. We will show that $\mathcal F$ is satisfiable if and only if $G$ is not $(p,q)$-biclique-Helly.

Our first claim is that every biclique of $G$ contained in $U$ and having at least six vertices is a stable set. Suppose, for a contradiction, that there is a biclique $B$ of $G$ contained in $U$ having at least six vertices but which is not a stable set. In particular, $B$ induces a connected subgraph of $G$. Hence, since each variable occurs in at most four clauses of $\mathcal F$, there must be two vertices $u,u'\in B\cap U_i$ for some $i\in\{1,\ldots,m\}$ and each of $u$ and $u'$ has at most three neighbors in $B-U_i$. Clearly, $u$ and $u'$ are the only vertices of $U_i$ in $B$ because $U_i$ is a clique. Since $B$ has at least six vertices, each of $u$ and $u'$ has at least one neighbor in $B-U_i$ and at least one of $u$ and $u'$ has at least two neighbors in $B-U_i$. Hence, by symmetry, there are two different $j,k\in\{1,\ldots,m\}-\{i\}$ such that $u$ has a neighbor $u''\in B\cap U_j$ and $u'$ has a neighbor $u'''\in B\cap U_k$. Since $B$ is a biclique, $u''$ and $u'''$ are adjacent. Thus, by construction, the literal associated with $u$ is the negation of the literal associated with $u''$, which in turn is the negation of the literal associated with $u'''$, which in turn is the negation of the literal associated with $u'$. We conclude that the literal associated with $u$ is the negation of the literal associated with $u'$, which contradicts the assumption that no clause of $\mathcal F$ contains a literal and its negation. This contradiction proves the claim.

Our second claim is that every maximal biclique $B$ of $G$ having at least six vertices contains $Q$. Suppose, for a contradiction, that $B$ is a maximal biclique of $G$ having at least six vertices and not containing $Q$. Since the vertices of $Q$ are pairwise false twins, the maximality of $B$ implies that if any vertex of $Q$ belongs to $B$ then all the vertices of $Q$ belong to $B$. Hence, no vertex of $Q$ belongs to $B$ or, equivalently, $B\subseteq Y\cup U$. Since $Y$ is a clique, $B$ contains at most two vertices of $Y$. Suppose first that there are two different vertices $y,y'\in B\cap Y$. Since $B$ has at least six vertices, there are three pairwise different vertices $u,u',u''\in B\cap U$ and let $i,j,k\in\{1,\ldots,m\}$ such that $u\in U_i$, $u'\in U_j$, and $u''\in U_k$. On the one hand, if $i$, $j$, and $k$ are pairwise different, then at least one of $u$, $u'$, and $u''$ induces a triangle with $y$ and $y'$. On the other hand, if $i=j$, then at least one of $y$ and $y'$ induces a triangle with $u$ and $u'$. By symmetry, these contradictions show that $B$ contains at most one vertex of $Y$. Suppose now that $B\cap Y=\{y_i\}$ for some $i\in\{1,\ldots,m\}$. If $y_i$ were complete to $B\cap U$, then every vertex of $Q$ would have the same neighbors in $B$ as $y_i$ and, since $Q$ is a stable set, $B\cup Q$ would be a biclique of $G$, contradicting the maximality of $B$. Thus, there is some vertex $u\in B\cap U_i$. Moreover, $u$ is the only vertex of $U_i$ in $B$ since any two different vertices of $U_i$ together with $y_i$ induce $\overline{P_3}$ in $G$. Hence, $(B\cap U)-\{u\}\subseteq U-U_i$ and, as a consequence, $y_i$ is complete to $(B\cap U)-\{u\}$. Since $B$ is a biclique, also $u$ is complete to $(B\cap U)-\{u\}$ and, by construction, all the literals associated with vertices of $B\cap U$ have the same variable. Since each variable occurs in at most four clauses of $\mathcal F$, $\vert B\cap U\vert\leq 4$ and $B=\{y_i\}\cup(B\cap U)$ has at most five vertices. This contradiction shows that $B\cap Y=\emptyset$ or, equivalently, $B\subseteq U$. Since $B$ contains no vertex of $Q$, the maximality of $B$ implies that $B$ is not a stable set, contradicting the first claim. This contradiction proves our second claim.

Suppose that there is a truth assignment $\mathcal A$ satisfying all clauses of $\mathcal F$. Thus, there is a set $W=\{w_1,\ldots,w_m\}$ such that $w_i\in U_i$ and the literal associated with vertex $w_i$ is satisfied by $\mathcal A$ for each $i\in\{1,\ldots,m\}$. By construction, $W$ is an independent set of $G$ and clearly $\{y_i\}\cup(W-\{w_i\})\cup Q$ is a maximal biclique of $G$ for each $i\in\{1,\ldots,m\}$. Furthermore, these $p+1$ maximal bicliques are $(p,q)$-intersecting and have a $(q-1)$-core. Hence, $G$ is not $(p,q)$-biclique-Helly.

Conversely, suppose that $G$ is not $(p,q)$-biclique-Helly. Let $\mathcal H$ be a $(p,q)$-intersecting partial hypergraph of the biclique hypergraph of $G$ having a $q^-$-core. Denote $E({\cal H})=\{B_1,\ldots,B_{p'}\}$ where, necessarily, $p'\geq p+1$. Since $p+1=m\geq 7$, we have that $p'\geq 7$. We assume, without loss of generality, that $\mathcal H$ is minimal in the sense that $\mathcal H-B$ has a $q^+$-core for each edge $B$ of $\mathcal H$. Let $W=\{w_1,\ldots,w_{p'}\}$ be a set such that, for each $i\in\{1,\ldots,p'\}$, $w_i$ belongs to the core of $\mathcal H-B_i$ but not to the core of $\mathcal H$. It is clear from the construction that the vertices $w_1,\ldots,w_{p'}$ are pairwise different. For each $i\in\{1,\ldots,p'\}$, the maximal biclique $B_i$ contains $Q$ because $B_i$ contains $W-\{w_i\}$ which has $p'-1\geq 6$ vertices and our second claim. Hence, $Q\subseteq\core(\mathcal H)$ and, consequently, $W\cap Q=\emptyset$. Notice that, for each $i\in\{1,\ldots,p'\}$, $W-\{w_i\}$ is a biclique of $G$ because it is a subset of $B_i$. Since $\vert W\vert=p'\geq 7$, Remark~\ref{rmk:coP3-C3} implies that $W$ is a biclique of $G$.

We claim that $U$ contains some stable set $S$ of cardinality at least $m$. If $W\cap Y=\emptyset$, then $W$ is a biclique of $G$ contained in $U$ having $p'\geq m\geq 7$ vertices and, by our first claim, $W$ is a stable set. Thus, if $W\cap Y=\emptyset$, then the claim holds by letting $S=W$. Hence, we assume, without loss of generality, that $W\cap Y\neq\emptyset$. Since $Y$ is a clique, $\vert W\cap Y\vert\leq 2$. Moreover, $\vert W\cap Y\vert\neq 2$ since otherwise the two vertices in $W\cap Y$ would induce a triangle with some of the at least five vertices in $W\cap U$. Therefore, $\vert W\cap Y\vert=1$ and suppose, without loss of generality, that $W\cap Y=\{y_1\}$. Suppose, for a contradiction, that $W\cap U_1\neq\emptyset$ and let $w\in W\cap U_1$. Clearly, $w$ is the only vertex of $U_1$ in $W$ since any two different vertices of $U_1$ together with $y_1$ induce $\overline{P_3}$ in $G$. Since $y_1$ is nonadjacent to $w$ and complete to $U-U_1$, $W\cap(U-U_1)$ is a stable set all whose vertices are adjacent to $w$. By construction, all the literals associated with the $p'-1\geq 6$ vertices of $W\cap U$ have the same variable, which contradicts the assumption that each variable occurs in at most four clauses of $\mathcal F$. This contradiction shows that $W\cap U_1=\emptyset$. Thus, $y_1$ is complete to $W\cap U$. Hence, $W\cap U$ is a stable set and, consequently, each $W\cap U_i$ contains at most one vertex for each $i\in\{2,\ldots,m\}$. Moreover, since $p'\geq m$, necessarily $\vert W\cap U_i\vert=1$ for each $i\in\{2,\ldots,m\}$ and $\vert W\vert=m$. Without loss of generality, suppose that $W\cap U_i=\{w_i\}$ for each $i\in\{2,\ldots,m\}$ and, consequently, $y_1=w_1$. By construction, the edge $B_1$ of $\mathcal H$ is a maximal biclique of $G$ containing $W'=\{w_2,\ldots,w_m\}$ but not containing $y_1$. Moreover, $B_1\cap Y=\emptyset$, since for each $i\in\{2,\ldots,m\}$, the vertices $y_i$ and $w_i$ together with any vertex of $W'-\{w_i\}$ induce $\overline{P_3}$ in $G$. Furthermore, $B_1\cap U_i=\{w_i\}$ for each $i\in\{2,\ldots,m\}$ because the number of neighbors in $W'$ of each vertex of $U_i-\{w_i\}$ is at least one (namely $w_i$) and at most four (because each variable occurs in at most four clauses of $\mathcal F$). Recall from the preceding paragraph that $Q\subseteq B_i$ for each $i\in\{1,\ldots,p'\}$. We conclude that $B_1=(B_1\cap U_1)\cup W'\cup Q$. Hence, since $B_1$ is a maximal biclique and $y_1\notin B_1$, necessarily $B_1\cap U_1\neq\emptyset$ and, since $U_1$ is a clique and each vertex of $U_1$ has at most three neighbors in $W'$, there is exactly one vertex $w^*$ in $B_1\cap U_1$. Therefore, $S=\{w^*,w_2,\ldots,w_m\}$ is a stable set contained in $U$ of cardinality $m$. This completes the proof of the claim.

Let $S$ be a stable set of $G$ contained in $U$ having cardinality at least $m$. Since $U=U_1\cup\cdots\cup U_m$ where $U_1,\ldots,U_m$ are cliques, necessarily $\vert S\vert=m$ and $\vert S\cap U_i\vert=1$ for each $i\in\{1,\ldots,m\}$. Hence, by construction, the $m$ literals associated with the vertices in $S$ are from pairwise different clauses and none of them is the negation of another one, which implies that $\mathcal F$ is satisfiable. This completes the proof of the correctness of the reduction and thus of the co-NP-completeness.\end{proof}

\begin{thm}\label{thm:pq-BH-coNP-qvar} The recognition problems of $(p,q)$-biclique-Helly graphs, for positive integers $p$ and $q$, is co-NP-complete if $q$ is part of the input.\end{thm}
\begin{proof} By Lemma~\ref{lem:pqBH-coNP}, the recognition problem belongs to co-NP. As in the proof of Theorem~\ref{thm:pq-BH-coNP-pvar}, we use a reduction from 3,4-SAT. Let $\mathcal F=\{C_1,\ldots,C_m\}$ be an instance of 3,4-SAT where $C_i=\{\ell_{i,1},\ell_{i,2},\ell_{i,3}\}$ for each $i\in\{1,\ldots,m\}$. We assume, without loss of generality, that $m\geq 6$ and no clause contains a literal and its negation.

Let $G$ be the graph with $(3m+2)(p+1)$ vertices defined as follows. The vertex set of $G$ is $X\cup Y\cup U$ where $X=\{x_1,\ldots,x_{p+1}\}$ is a clique, $Y=\{y_1,\ldots,y_{p+1}\}$ is a stable set, and $U=U_1\cup\cdots\cup U_m$ is such that, for each $i\in\{1,\ldots,m\}$, $U_i$ is the union of the pairwise complete stable sets $U_{i,1}$, $U_{i,2}$, and $U_{i,3}$ of cardinality $p+1$ each. Moreover, for each $i,j\in\{1,\ldots,m\}$, $x_i$ is adjacent to $y_j$ if and only if $i\neq j$. Furthermore, $U$ is complete to $X$ and anticomplete to $Y$ and, for each two different $i,j\in\{1,\ldots,m\}$ and each $a,b\in\{1,2,3\}$, $U_{i,a}$ is complete (resp.\ anticomplete) to $U_{j,b}$ if and only if the literal $\ell_{i,a}$ is (resp.\ is not) the negation of $\ell_{j,b}$. For each $i\in\{1,\ldots,m\}$ and each $a\in\{1,2,3\}$, we call $\ell_{i,a}$ \emph{the literal associated with the vertices of} $U_{i,a}$. Let $q = m(p+1)+1$. We will show that $\mathcal F$ is satisfiable if and only if $G$ is not $(p,q)$-biclique-Helly.

Suppose that $\mathcal A$ is a truth assignment satisfying all clauses of $\mathcal F$. Thus, there exists a set $A=\{\ell_{1,a_1},\ldots,\ell_{m,a_m}\}$ of $m$ literals satisfied by $\mathcal A$ where $a_1,\ldots,a_m\in\{1,2,3\}$. Let $W$ be the set of $m(p+1)$ vertices of $U$ whose associated literals are in $A$. It is clear by construction that $W$ is an independent set of $G$. Clearly, each set of the form $W\cup\{x_i\}\cup(Y-\{y_i\})$, for any $i\in\{1,\ldots,p+1\}$, is a maximal biclique of $G$. Furthermore, these $p+1$ sets are $(p,q)$-intersecting with $(q-1)$-core. Hence, $G$ is not $(p,q)$-biclique-Helly.

Conversely, suppose that $G$ is not $(p,q)$-biclique-Helly. Let $\mathcal H$ be a $(p,q)$-intersecting partial hypergraph of the biclique hypergraph of $G$ having a $(q-1)^-$-core. Denote $E(\mathcal H)=\{B_1,\ldots,B_{p'}\}$. We assume, without loss of generality, that $\mathcal H$ is minimal in the sense that $\mathcal H-B$, for every biclique $B$ of $\mathcal H$, has $q^+$-core. In particular, $B_1$ is a maximal biclique of $G$ having more than $q-1=m(p+1)$ vertices. Clearly, $B_1$ contains no two different vertices $x_i,x_j$ of $X$, since otherwise the fact that $x_i$ and $x_j$ are adjacent and $N_G(x_i)-\{y_j\}=N_G(x_j)-\{y_i\}$ would imply that at most two other vertices, namely $y_i$ and $y_j$, may belong to $B_1$. Hence, $B_1$ contains at most one vertex of $X$ and, by construction, contains at most $p+1$ vertices of $X\cup Y$. Since $B_1$ has more than $m(p+1)$ vertices, $B_1\cap U$ has more than $(m-1)(p+1)$ vertices. Since $\varPi=\{U_{i,a}\}_{1\leq i\leq m,1\leq a\leq 3}$ is a partition of $U$ into $3m$ sets consisting of $p+1$ pairwise false twins of $G$ each, the maximality of $B_1$ implies that $B_1\cap U$ is the union of some members of $\varPi$. Let $U^*=\{u_{i_1,a_1},\ldots,u_{i_z,a_z}\}$ be a set consisting of exactly one representative of each member of $\varPi$ contained in $B_1\cap U$ and such that $u_{i_s,a_s}\in U_{i_s,a_s}$ for each $s\in\{1,\ldots,z\}$. Since $B_1\cap U$ has more than $(m-1)(p+1)$ vertices, $U^*$ consists of at least $m\geq 6$ vertices. Hence, reasoning as in the first claim of the proof of Theorem~\ref{thm:pq-BH-coNP-pvar}, $U^*$ is necessarily a stable set. By construction, this means that $\vert U\vert=m$ and the literals $\ell_{i_1,a_1},\ldots,\ell_{i_m,a_m}$ are from pairwise different clauses and none of them is the negation of another one. This proves that $\mathcal F$ is satisfiable. This completes the proof of the correctness of the reduction and thus of the co-NP-completeness.\end{proof}

\subsection{Hereditary $(p,q)$-biclique-Helly graphs}

The remaining of this section is devoted to the problems of characterizing and recognizing hereditary $(p,q)$-biclique-Helly hypergraphs. For that purpose, we define the \emph{bioculars} as follows. For each pair of positive integers $p$ and $q$ and each $s\in\{0,\ldots,q-1\}$ such that $(p,q)\neq(1,1)$, a \emph{$(p+1,q,s)$-biocular} is a graph whose vertex set is the union of two disjoint sets $U$ and $W$ where $U$ is a $((p+1)(q-s)+s)$-set and $T_1,\ldots,T_{p+1}$ are pairwise disjoint $(q-s)$-subsets of $U$ such that one of the following statements holds:
\begin{enumerate}
 \item[(\bica)] $p\in\{1,2\}$, $W=\emptyset$, and $U-T_i$ is a biclique but $(U-T_i)\cup\{v_i\}$ is not a biclique for each $v_i\in T_i$ for each $i\in\{1,\ldots,p+1\}$, and either $p=1$ or $s=0$;
 \item[(\bicb)] $p\geq 2$, $(p,q)\neq(2,1)$, $W=\{w_1,\ldots,w_{p+1}\}$, $U$ is a biclique, and $(U-T_i)\cup\{w_i\}$ is a biclique but $(U-T_i)\cup\{w_i,v_i\}$ is not a biclique for each $v_i\in T_i$ for each $i\in\{1,\ldots,p+1\}$.
\end{enumerate}
If $p\geq 2$ then the vertices of $W$ may induce in $G$ an arbitrary graph. For $(p,q)=(1,1)$, we define the \emph{$(2,1,0)$-bioculars} as the graphs $\overline{P_3}$ and $K_3$. We are now ready to prove that the $(p+1,q,s)$-bioculars are not $(p,q)$-biclique-Helly.

\begin{lem}\label{lem:biocular} If $p$ and $q$ are positive integers and $s\in\{0,\ldots,q-1\}$, then no $(p+1,q,s)$-biocular is $(p,q)$-biclique-Helly.\end{lem}
\begin{proof} Clearly, $\overline{P_3}$ and $K_3$ are not $(1,1)$-biclique-Helly. Hence, we assume, without loss of generality, that $(p,q)\neq(1,1)$. Let $G$ be a $(p+1,q,s)$-biocular and let $U$, $W$, and $T_1,\ldots,T_{p+1}$ as in the corresponding definition. If condition (\bica) of the definition of $(p+1,q,s)$-bioculars holds, then $\{U-T_1,\ldots,U-T_{p+1}\}$ is a $(p,q)$-intersecting family of maximal bicliques of $G$ having $s$-core and, consequently, $G$ is not $(p,q)$-biclique-Helly. Therefore, from now on we assume, without loss of generality, that condition (\bicb) holds.

Let $i\in\{1,\ldots,p+1\}$ and let $B_i$ be any maximal biclique containing $(U-T_i)\cup\{w_i\}$. We claim that $B_i=(U-T_i)\cup\{w_i\}$. Since $(U-T_i)\cup\{w_i,v_i\}$ is not a biclique for any $v_i\in T_i$, $B_i\cap U=U-T_i$.  It only remains to prove that $B_i\cap W=\{w_i\}$. Suppose for a contradiction, that $w_j\in B_i\cap W$ for some $j\in\{1,\ldots,p+1\}-\{i\}$. Lemma~\ref{lem:B1B2} applied to $B=(U-T_j)\cup\{v_j\}$, $B'=(U-T_j)\cup\{w_j\}$, and $v=v_j$ implies that $\{v_j,x,w_j\}$ is not a biclique for any $x\in U-T_j$. In particular, if $k\in\{1,\ldots,p+1\}-\{i,j\}$ (which is possible because $p\geq 2$), $S=\{v_j,w_j,v_k\}$ is not a biclique for any $v_k\in T_k$, which contradicts the fact that $S$ is a subset of the biclique $B_i$. This contradiction proves that $B_i\cap W=\{w_i\}$ and we conclude that $B_i=(U-T_i)\cup\{w_i\}$. Since the maximal bicliques $B_1,\ldots,B_{p+1}$ are $(p,q)$-intersecting but have $s$-core, $G$ is not $(p,q)$-biclique-Helly.\end{proof}

Let $G$ be a graph. A \emph{biclique-matrix} $B(G)$ of $G$ is an incidence matrix of $\mathcal B(G)$. Recall that, by definition, a graph $G$ is hereditary $(p,q)$-biclique-Helly if and only if $\mathcal B(G')$ is $(p,q)$-Helly for every induced subgraph $G'$ of $G$. It is easy to see that if $G'$ is an induced subgraph of $G$, $\mathcal B(G')$ is the hypergraph formed by the inclusion-wise maximal edges of the subhypergraph of $\mathcal B(G)$ induced by $V(G')$. For this reason, not every partial subhypergraph of $\mathcal B(G)$ is the biclique hypergraph $\mathcal B(G')$ of some induced subgraph of $G'$ of $G$. We will say that $G$ is \emph{strong $(p,q)$-biclique-Helly} if $\mathcal B(G)$ is strong $(p,q)$-Helly.

The analogue of Theorem~\ref{thm:HpqCHelly} for hereditary $(p,q)$-biclique-Helly graphs is the following.

\begin{thm}\label{thm:HpqBHelly} If $p$ and $q$ are positive integers, then the following statements are equivalent for each graph $G$:
\begin{enumerate}[(i)]
 \item\label{it:HpqBH1} $G$ is hereditary $(p,q)$-biclique-Helly;
 \item\label{it:HpqBH2} $G$ is $(p,q')$-biclique-Helly for each $q'\geq q$;
 \item\label{it:HpqBH3} $G$ is strong $(p,q)$-biclique-Helly;
 \item\label{it:HpqBH4} Each family of $p+1$ maximal bicliques of $G$ is strong $(p,q)$-Helly;
 \item\label{it:HpqBH5} $\varPhi_q(\mathcal B(G))$ is hereditary $p$-Helly;
 \item\label{it:HpqBH8} $B(G)$ contains no incidence matrix of $\mathcal J_{p+1,q,s}$ as a submatrix for any $s\in\{0,
 \ldots,q-1\}$.
 \item\label{it:HpqBH7} for each $s\in\{0,\ldots,q-1\}$, each $((p+1)(q-s)+s)$-subset $U$ of $V(G)$, and each $p+1$ pairwise disjoint $(q-s)$-subsets $T_1,\ldots,T_{p+1}$ of $U$ such that $U-T_1,\ldots,U-T_{p+1}$ are bicliques of $G$, there is some $i\in\{1,\ldots,p+1\}$ and some $v\in T_i$ such that $(U-T_i)\cup\{v,w\}$ is a biclique for every vertex $w$ of $G$ such that $(U-T_i)\cup\{w\}$ is a biclique;
 \item\label{it:HpqBH6} $G$ contains no induced $(p+1,q,s)$-biocular for any $s\in\{0,\ldots,q-1\}$.
\end{enumerate}\end{thm}
\begin{proof} We can prove that (\ref{it:HpqBH7}) is equivalent to statement (\ref{it:HpqH8}) of Theorem~\ref{thm:HpqHelly} applied to $\mathcal H=\mathcal B(G)$, in a way entirely analogous to that used to prove the second claim within the proof of Theorem~\ref{thm:HpqCHelly}, by letting Lemma~\ref{lem:BP} play the role of Lemma~\ref{lem:CP}. Hence, Theorem~\ref{thm:HpqHelly} applied to $\mathcal H=\mathcal B(G)$ implies that statements (\ref{it:HpqBH2}) to (\ref{it:HpqBH8}) are all equivalent. Hence, since (\ref{it:HpqBH3})${}\Rightarrow{}$(\ref{it:HpqBH1}) follows by an entirely analogous argumentation as the one used in the proof of (\ref{it:HpqCH3})${}\Rightarrow{}$(\ref{it:HpqCH1}) of Theorem~\ref{thm:HpqCHelly} and (\ref{it:HpqBH1})${}\Rightarrow{}$(\ref{it:HpqBH6}) follows from Lemma~\ref{lem:biocular}, in order to complete the proof of the theorem it suffices to show that (\ref{it:HpqBH6})${}\Rightarrow{}$(\ref{it:HpqBH7}), which we do below.

\mbox{(\ref{it:HpqBH6})${}\Rightarrow{}$(\ref{it:HpqBH7})} Suppose that there is some $s\in\{0,\ldots,q-1\}$, some $((p+1)(q-s)+s)$-subset $U$ of $V(G)$ and $p+1$ pairwise disjoint $(q-s)$-subsets $T_1,\ldots,T_{p+1}$ of $U$ such that $U-T_1,\ldots,U-T_{p+1}$ are bicliques of $G$ and such that for each $i\in\{1,\ldots,p+1\}$ and each $v\in T_i$ there is some vertex $\trans{v}$ of $G$ such that $(U-T_i)\cup\{\trans{v}\}$ is a biclique but $(U-T_i)\cup\{v,\trans{v}\}$ is not a biclique of $G$.

Suppose first that $p=1$. If $(U-T_i)\cup\{v\}$ is not a biclique for each $i\in\{1,2\}$ and each $v\in T_i$, then $U$ induces a $(p+1,q,s)$-ocular in $G$. If, on the contrary, there is some $i\in\{1,2\}$ and some $v\in T_i$ such that $(U-T_i)\cup\{v\}$ is a biclique, then Lemma~\ref{lem:B1B2} applied to the bicliques $(U-T_i)\cup\{v\}$ and $(U-T_i)\cup\{\trans{v}\}$ implies that $\{v,b,\trans{v}\}$ is not a biclique for any $b\in U-T_i$ and, consequently, $U'\cup\{v,\trans{v}\}$ induces a $(p+1,q,q-1)$-ocular in $G$ for any $(q-1)$-subset $U'$ of $U-T_i$. In either case, (\ref{it:HpqBH6}) does not hold. Thus, from now on, we assume that $p\geq 2$.

Suppose now that $U$ is not a biclique of $G$ and let $S$ be $3$-subset of $V(G)$ which is not a biclique. Since $U-T_1,\ldots,U-T_{p+1}$ are bicliques of $G$, there must be at least one element of $S$ in each of $T_1,\ldots T_{p+1}$. Hence, $p=2$ and $S=\{s_1,s_2,s_3\}$ for some $s_i\in T_i$ for each $i\in\{1,2,3\}$. Let $Z=U-(T_1\cup T_2\cup T_3)$. If there were some vertex $z\in Z$, then the fact that $\{s_1,s_2,s_3\}$ is a not a biclique of $G$ but $\{z,s_i,s_j\}$ is a biclique of $G$ (because it is contained in $U-T_k$) for each permutation $i,j,k$ of $1,2,3$, would contradict Lemma~\ref{lem:abcd}. Hence, $Z=\emptyset$; i.e., $s=0$. Let $i\in\{1,2,3\}$ and let $v\in T_i$. If $j$ and $k$ are such that $i,j,k$ is a permutation of $1,2,3$, then the fact that $\{v,s_i,s_j\}$ and $\{v,s_i,s_k\}$ are bicliques of $G$ (because they are contained in $U-T_k$ and $U-T_j$, respectively) but $\{s_1,s_2,s_3\}$ is not a biclique of $G$ implies, by virtue of Lemma~\ref{lem:abcd}, that $\{v,s_j,s_k\}$ is not a biclique of $G$ and, in particular, $(U-T_i)\cup\{v\}$ is not a biclique of $G$. This proves that $U$ induces a $(p,q,0)$-ocular in $G$.

It only remains to consider the case where $p\geq 2$ and $U$ is a biclique of $G$. For each $i\in\{1,\ldots,p+1\}$ and each $v\in T_i$, let $W(v)=\{w\in V(G):\,(U-T_i)\cup\{w\}\text{ is a biclique of $G$ but }(U-T_i)\cup\{v,w\}$ is not a biclique of $G\}$. Suppose first that for some $i\in\{1,\ldots,p+1\}$, there are two $x,y\in T_i$ such that $W(x)$ and $W(y)$ are inclusion-wise incomparable; i.e., there are two vertices $x',y'$ of $G$ such that $(U-T_i)\cup\{x,x'\}$ and $(U-T_i)\cup\{y,y'\}$ are not bicliques of $G$ but $(U-T_i)\cup\{x,y'\}$ and $(U-T_i)\cup\{y,x'\}$ are bicliques of $G$. By symmetry, let $i=1$. Thus, for any choice of $v_i\in T_i$ for each $i\in\{2,\ldots,p+1\}$, the sets $U'=\{x,y,v_2,\ldots,v_p\}\cup(T_{p+1}-\{v_{p+1}\})\cup Z$ and $W'=\{x',y',\trans{v_2},\ldots,\trans{v_p}\}$ together induce a $(p+1,q,q-1)$-biocular in $G$; in fact: (1) $(U'-\{x\})\cup\{x'\}$ is a biclique of $G$ (because it is contained in $(U-T_i)\cup\{y,x'\}$) and $(U'-\{x\})\cup\{x,x'\}$ is not a biclique of $G$ (by Lemma~\ref{lem:B1B2} applied to the bicliques $(U-T_i)\cup\{x\}$ and $(U-T_i)\cup\{x'\}$); (2) analogously, $(U'-\{y\})\cup\{y'\}$ is a biclique of $G$ but $(U'-\{y\})\cup\{y,y'\}$ is not a biclique of $G$; and (3) $(U'-\{v_i\})\cup\{\trans{v_i}\}$ is a biclique of $G$ (because it is contained in $(U-T_i)\cup\{\trans{v_i}\}$) but $(U'-\{v_i\})\cup\{v_i,\trans{v_i}\}$ is not a biclique of $G$ (by Lemma~\ref{lem:B1B2} applied to the bicliques $(U-T_i)\cup\{v_i\}$ and $(U-T_i)\cup\{\trans{v_i}\}$) for each $i\in\{2,\ldots,p\}$. Thus, (\ref{it:HpqBH6}) does not hold.

Suppose that, on the contrary, for each $i\in\{1,\ldots,p+1\}$, the family $\{W(v):\,v\in T_i\}$ is a chain under inclusion, let $v_i\in T_i$ such that $W(v_i)$ is the minimum element of the chain, and let $w_i=\trans{v_i}$. By construction, $(U-T_i)\cup\{w_i\}$ is a biclique of $G$ but $(U-T_i)\cup\{v,w_i\}$ is not a biclique of $G$ for each $i\in\{1,\ldots,p+1\}$ and each $v\in T_i$. Hence, if $W=\{w_1,\ldots,w_{p+1}\}$, then $U\cup W$ induces a $(p+1,q,s)$-ocular subgraph in $G$ and (\ref{it:HpqBH6}) does not hold. This completes the proof of (\ref{it:HpqBH6})${}\Rightarrow{}$(\ref{it:HpqBH7}) and thus of the theorem.\end{proof}

By combining the equivalence (\ref{it:HpqBH1})${}\Rightarrow{}$(\ref{it:HpqBH3}) of the above theorem with the equivalence (\ref{it:HpqH1})${}\Leftrightarrow{}$(\ref{it:HpqH3}) of Theorem~\ref{thm:HpqHelly}, we obtain the following result analogous to Corollary~\ref{cor:HpqCHelly}, which is likewise crucial for our derivation of a polynomial-time recognition algorithm.

\begin{cor}\label{cor:HpqBHelly} If $p$ and $q$ are positive integers, then a graph $G$ is hereditary $(p,q)$-biclique-Helly if and only if $\mathcal B(G)$ is hereditary $(p,q)$-Helly.\end{cor}

From the above result, the fact that $\mathcal B(G)$ is simple, and Lemma~\ref{lem:H1qHelly}, the following characterizations of hereditary $(1,q)$-biclique-Helly graphs follows.

\begin{cor}\label{cor:H1qBHelly} If $q$ is a positive integer, then the following statements are equivalent for each graph $G$:
\begin{enumerate}[(i)]
 \item\label{it:H1qBH1} $G$ is hereditary $(1,q)$-biclique-Helly
 \item\label{it:H1qBH2} $G$ has at most one maximal biclique of cardinality at least $q$.
 \item\label{it:H1qBH3} the union of all the $q$-bicliques of $G$ is empty or a maximal biclique of $G$.
 \item\label{it:H1qBH4} the union of all the $q$-bicliques of $G$ is a biclique of $G$.
\end{enumerate}\end{cor}

We now address the problem of recognizing hereditary $(p,q)$-biclique-Helly graphs, providing different time bounds which are polynomial for fixed $p$ and $q$.

\begin{thm}\label{thm:recog-pq-HBH} If $q$ is any fixed positive integer, then the recognition problem of hereditary $(p,q)$-biclique-Helly graphs, where $p$ is part of the input, can be solved in:
\begin{enumerate}[(i)]
\item\label{it:pqHBH1-rec} $O(m+n)$ time if $(p,q)=(1,1)$, $(1,2)$ or $(2,1)$;
\item\label{it:pqHBH2-rec} $O(qn^q+m+n)$ time if $p=1$ and $q\geq 3$;
\item\label{it:pqHBH3-rec} $O\bigl(qn^q+(\psi N_q+m+n)\binom{N_q}p\bigr)$ time if $p\geq 2$, where $N_q$ is the number of $q$-bicliques of the input graph.
\end{enumerate}
In particular, if $p$ and $q$ are both fixed, bounds (\ref{it:pqHBH2-rec}) and (\ref{it:pqHBH3-rec}) above become $O(n^q)$ and $O(\psi n^{(p+1)q})$, respectively.\end{thm}
\begin{proof} Let $G$ be the input graph and let $\mathcal B$ denote its biclique hypergraph.

Lemma~\ref{lem:11,12,21} implies that if $(p,q)=(1,1)$, $(1,2)$, or $(2,1)$, then $G$ is hereditary $(p,q)$-biclique-Helly if and only if $G$ is a complete bipartite graph. This proves the validity of bound (\ref{it:pqHBH1-rec}).

Suppose that $p=1$ and $q\geq 3$. By Corollary~\ref{cor:H1qBHelly}, $G$ is hereditary $(p,q)$-biclique-Helly if and only if the union of all the $q$-bicliques of $G$ is a biclique of $G$. Thus, the bound (\ref{it:pqHBH2-rec}) follows from the facts that we can determine all the $q$-bicliques of $G$ and compute their union in $O(qn^q+n)$ time and determine whether $U$ is a biclique of $G$ in additional $O(m+n)$ time.

Suppose now that $p\geq 2$. By Corollary~\ref{cor:HpqBHelly}, $G$ is hereditary $(p,q)$-clique-Helly if and only if $\mathcal B(G)$ is hereditary $(p,q)$-Helly, and the latter can be verified in $O\bigl(qn^q+(\psi N_q+m+n)\binom{N_q}p\bigr)$ time by virtue of Lemma~\ref{lem:technical2} because the $q$-bicliques of $G$ can be enumerated in $O(qn^q)$ time and the core of $\mathcal B_P$ for any subset $P$ of $V(G)$ can be computed in $O(m+n)$ time (by Lemma~\ref{lem:BP}).\end{proof}

Finally, we prove that if $p$ or $q$ is part of the input then the recognition of hereditary $(p,q)$-biclique-Helly graphs is co-NP-complete.

\begin{thm} The recognition problem of hereditary $(p,q)$-biclique-Helly graphs, for positive integers $p$ and $q$, is co-NP-complete if $p$ or $q$ is part of the input.\end{thm}
\begin{proof} Let $G$ be the input graph. The recognition problem is in co-NP because, by Theorem~\ref{thm:HpqBHelly}, if $G$ is a graph which is not hereditary $(p,q)$-biclique-Helly, then there is a certificate in the form of an induced subgraph of $G$ that is a $(p,q,s)$-biocular graph for some $s\in\{0,\ldots,q-1\}$.

In order to prove the NP-completeness if $p$ or $q$ is part of the input, we use the same reductions from 3,4-SAT used in Theorem~\ref{thm:pq-BH-coNP-pvar} or \ref{thm:pq-BH-coNP-qvar}, respectively. Let $\mathcal F=\{C_1,\ldots,C_m\}$ be an instance of 3,4-SAT and let $G$ be the graph constructed from $\mathcal F$ as in Theorem~\ref{thm:pq-BH-coNP-pvar} or \ref{thm:pq-BH-coNP-qvar}, respectively. It remains to show that $\mathcal F$ is satisfiable if and only if $G$ is not hereditary $(p,q)$-biclique-Helly. The necessity is direct, since $G$ is an induced subgraph of itself. For the sufficiency, it is enough to observe that the constructions used in the proof of the sufficiency in Theorems~\ref{thm:pq-BH-coNP-pvar} and \ref{thm:pq-BH-coNP-qvar} can be applied to any non-$(p,q)$-biclique-Helly induced subgraph $G'$ of $G$ for obtaining a truth assignment satisfying all clauses of $\mathcal F$.\end{proof}

\section*{Acknowledgments}

M.C.~Dourado was partially supported by Conselho Nacional de Desenvolvimento Cient\'ifico e Tecnol\'ogico, Brazil, Grant number 305404/2020-2. L.N.~Grippo and M.D.~Safe were partially supported by ANPCyT PICT 2017-1315. M.D.~Safe was partially supported by Universidad Nacional del Sur Grant PGI L24/115.

\end{document}